\documentclass[12pt]{amsart}
\usepackage{amsmath}
\usepackage{amsfonts}
\begin{document}
\numberwithin{equation}{section}

\def\1#1{\overline{#1}}
\def\2#1{\widetilde{#1}}
\def\3#1{\widehat{#1}}
\def\4#1{\mathbb{#1}}
\def\5#1{\frak{#1}}
\def\6#1{{\mathcal{#1}}}

\newcommand{\de}{\partial}
\newcommand{\R}{\mathbb R}
\newcommand{\Ha}{\mathbb H}
\newcommand{\al}{\alpha}
\newcommand{\tr}{\widetilde{\rho}}
\newcommand{\tz}{\widetilde{\zeta}}
\newcommand{\tk}{\widetilde{C}}
\newcommand{\tv}{\widetilde{\varphi}}
\newcommand{\hv}{\hat{\varphi}}
\newcommand{\tu}{\tilde{u}}
\newcommand{\tF}{\tilde{F}}
\newcommand{\debar}{\overline{\de}}
\newcommand{\Z}{\mathbb Z}
\newcommand{\C}{\mathbb C}
\newcommand{\Po}{\mathbb P}
\newcommand{\zbar}{\overline{z}}
\newcommand{\G}{\mathcal{G}}
\newcommand{\So}{\mathcal{S}}
\newcommand{\Ko}{\mathcal{K}}
\newcommand{\U}{\mathcal{U}}
\newcommand{\B}{\mathbb B}
\newcommand{\oB}{\overline{\mathbb B}}
\newcommand{\Cur}{\mathcal D}
\newcommand{\Dis}{\mathcal Dis}
\newcommand{\Levi}{\mathcal L}
\newcommand{\SP}{\mathcal SP}
\newcommand{\Sp}{\mathcal Q}
\newcommand{\A}{\mathcal O^{k+\alpha}(\overline{\mathbb D},\C^n)}
\newcommand{\CA}{\mathcal C^{k+\alpha}(\de{\mathbb D},\C^n)}
\newcommand{\Ma}{\mathcal M}
\newcommand{\Ac}{\mathcal O^{k+\alpha}(\overline{\mathbb D},\C^{n}\times\C^{n-1})}
\newcommand{\Acc}{\mathcal O^{k-1+\alpha}(\overline{\mathbb D},\C)}
\newcommand{\Acr}{\mathcal O^{k+\alpha}(\overline{\mathbb D},\R^{n})}
\newcommand{\Co}{\mathcal C}
\newcommand{\Hol}{{\sf Hol}(\mathbb H, \mathbb C)}
\newcommand{\Aut}{{\sf Aut}(\mathbb D)}
\newcommand{\D}{\mathbb D}
\newcommand{\oD}{\overline{\mathbb D}}
\newcommand{\oX}{\overline{X}}
\newcommand{\loc}{L^1_{\rm{loc}}}
\newcommand{\la}{\langle}
\newcommand{\ra}{\rangle}
\newcommand{\thh}{\tilde{h}}
\newcommand{\N}{\mathbb N}
\newcommand{\kd}{\kappa_D}
\newcommand{\Hr}{\mathbb H}
\newcommand{\ps}{{\sf Psh}}
\newcommand{\Hess}{{\sf Hess}}
\newcommand{\subh}{{\sf subh}}
\newcommand{\harm}{{\sf harm}}
\newcommand{\ph}{{\sf Ph}}
\newcommand{\tl}{\tilde{\lambda}}
\newcommand{\gdot}{\stackrel{\cdot}{g}}
\newcommand{\gddot}{\stackrel{\cdot\cdot}{g}}
\newcommand{\fdot}{\stackrel{\cdot}{f}}
\newcommand{\fddot}{\stackrel{\cdot\cdot}{f}}
\def\v{\varphi}
\def\Re{{\sf Re}\,}
\def\Im{{\sf Im}\,}


\def\Label#1{\label{#1}}


\def\cn{{\C^n}}
\def\cnn{{\C^{n'}}}
\def\ocn{\2{\C^n}}
\def\ocnn{\2{\C^{n'}}}
\def\je{{\6J}}
\def\jep{{\6J}_{p,p'}}
\def\th{\tilde{h}}


\def\dist{{\rm dist}}
\def\const{{\rm const}}
\def\rk{{\rm rank\,}}
\def\id{{\sf id}}
\def\aut{{\sf aut}}
\def\Aut{{\sf Aut}}
\def\CR{{\rm CR}}
\def\GL{{\sf GL}}
\def\Re{{\sf Re}\,}
\def\Im{{\sf Im}\,}
\def\U{{\sf U}}

\def\la{\langle}
\def\ra{\rangle}

\emergencystretch15pt \frenchspacing

\newtheorem{theorem}{Theorem}[section]
\newtheorem{lemma}[theorem]{Lemma}
\newtheorem{proposition}[theorem]{Proposition}
\newtheorem{corollary}[theorem]{Corollary}

\theoremstyle{definition}
\newtheorem{definition}[theorem]{Definition}
\newtheorem{example}[theorem]{Example}

\theoremstyle{remark}
\newtheorem{remark}[theorem]{Remark}
\numberwithin{equation}{section}

\title[Evolution families I:  the unit disc]{Evolution Families and the Loewner Equation I: the unit disc}
\author[F. Bracci]{Filippo Bracci}
\address{Dipartimento Di Matematica, Universit\`{a} Di Roma \textquotedblleft Tor
Vergata\textquotedblright, Via Della Ricerca Scientifica 1, 00133, Roma,
Italy. }
\email{fbracci@mat.uniroma2.it}
\author[M. D. Contreras]{Manuel D. Contreras}

\author[S. D\'{\i}az-Madrigal]{Santiago D\'{\i}az-Madrigal}
\address{Camino de los Descubrimientos, s/n\\
Departamento de Matem\'{a}tica Aplicada II\\
Escuela T\'{e}cnica Superior de Ingenieros\\
Universidad de Sevilla\\
Sevilla, 41092\\
Spain.}\email{contreras@us.es} \email{madrigal@us.es}
\date{\today }


\thanks{$^\dag$Partially supported by the \textit{Ministerio
de Ciencia y Innovaci\'on} and the European Union (FEDER),
project MTM2006-14449-C02-01, by \textit{La Consejer\'{\i}a de
Educaci\'{o}n y Ciencia de la Junta de Andaluc\'{\i}a} and by
the European Science Foundation Research Networking Programme
HCAA}

\begin{abstract}  In this paper we introduce a general version
of the Loewner differential equation which allows us to present
a new and unified treatment of both the radial equation
introduced in 1923 by K. Loewner and the chordal equation
introduced in 2000 by O. Schramm. In particular, we prove that
evolution families  in the unit disc are in one to one
correspondence with solutions to this new type of Loewner
equations. Also, we give a Berkson-Porta type formula for
non-autonomous weak holomorphic vector fields which generate
such Loewner differential equations and  study in detail
geometric and dynamical properties of evolution families.
\end{abstract}

\maketitle

\tableofcontents

\section{Introduction}

In 1923, Loewner \cite{Loewner} developed a machinery to
``embed'' a slit domain of the complex plane into a family of
domains endowed with a certain order. The key idea was to
represent such domains by means of a family (nowadays known as
a {\sl Loewner chain})  of univalent functions defined on the
unit disc and satisfying a suitable differential equation. Such
a machinery was then studied and extended to other types  of
simply connected domains  by  Kufarev in 1943 and Pommerenke in
1965 (see, {\sl e.g.}, \cite{Duren}, \cite{Pommerenke}, and
\cite{Rosenblum-Rovnyak}).

Since the original paper of Loewner, this method has shown to
be extremely useful when dealing with many different problems,
especially those having some character of extremality. In fact,
in 1984 de Branges used (extensions of) Loewner's theory  to
solve the Bieberbach conjecture.

The classical {\sl radial Loewner equation} in the unit disc
$\D:=\{\zeta\in \C: |\zeta|<1\}$ is the following
non-autonomous differential equation
\begin{equation}\label{loewnervero}
\begin{cases}
\overset{\bullet}{w}  =G(w,t)& \quad\text{for almost every }t\in\lbrack s,\infty)\\
w(s)   =z&
\end{cases}
\end{equation}
where $s\in[0,+\infty)$, $G(w,t)=-wp(w,t)$ with the function
$p:\mathbb{D}\times\lbrack0,+\infty)\rightarrow\mathbb{C}$
measurable in $t$, holomorphic in $z$, $p(0,t)=1$ for all
$t\geq 0$ and $\Re p(z,t)\geq 0$. In fact Loewner himself
studied the case when $p(z,t)=\frac {1+k(t)z}{1-k(t)z}$ for
some continuous function
$k:[0,+\infty)\rightarrow\partial\mathbb{D}$. Write $t\mapsto
\v_{s,t}(z)$ for the solution of such a differential equation.
Then for $0\leq s\leq t<+\infty$ the maps $\v_{s,t}$ are
holomorphic self maps of $\D$ which verify the following
properties:
\begin{enumerate}
\item $\varphi_{s,s}=id_{\mathbb{D}},$

\item $\varphi_{s,t}=\varphi_{u,t}\circ\varphi_{s,u}$ for all $0\leq s\leq
u\leq t<+\infty,$

\item $\varphi_{s,t}(0)=0$ and
$\varphi_{s,t}^{\prime}(0)=e^{s-t}$ for all $0\leq s\leq u\leq
t<+\infty$.
\end{enumerate}

We call such a family  $(\v_{s,t})$  an {\sl evolution family}
of the unit disc (see Definition \ref{def-ev} for a precise
definition).

The  hypotheses  $G(0,t)\equiv 0$ and $p(0,t)\equiv 1$, which
forces the evolution family $(\v_{s,t})$ to fix the origin and
to have normalized first derivatives at $0$, are strongly used
in the construction of the family itself (mainly in proving
semicompleteness and holomorphicity) because they allow to use
distortion theorems.

Until the end of the XX century, there were only few papers
where  equation \eqref{loewnervero} was studied assuming
$G(\tau,t)\equiv 0$ for some $\tau\in \de \D$. We cite the
pioneering works of Goryainov \cite{Gorjainov} and Goryainov
and Ba \cite{Goryainov-Ba}. After that,  Schramm \cite{Schramm}
and Lawler, Schramm and Werner \cite{Lawler-Schramm-Werner-I},
\cite{Lawler-Schramm-Werner-II} proved the Mandelbroit
conjecture using a stochastic version of this chordal Loewner
equation. Also Bauer \cite{Bauer}, Marshall and Rohde
\cite{Marshall-Rohde} and Prokhorov and Vasiliev
\cite{Prokhorov-Vasiliev} studied a similar chordal Loewner
equation. In such a case, $G(w,t)=(1-w)^{2}p(w,t)$ where
$p(w,t)=\frac{1}{g(w)+ih(t)}$ with $g(w)=\frac{1+w}{1-w}$ and
$h:[0,+\infty)\rightarrow\mathbb{R}$ continuous.  Solutions to
such an equation correspond to evolution families $(\v_{s,t})$
with boundary fixed points and are usually stated in the half
plane model.

In this paper we study general evolution families of the unit
disc. Our method allows to treat at the same time evolution
families with inner fixed points and with no interior fixed
points. In particular, we can solve \eqref{loewnervero} in case
of boundary fixed points {\sl without} assuming any particular
form of $G(z,t)$. More in detail, our aim is to completely
characterize evolution families by means of a differential
equation of type \eqref{loewnervero}. The key observation on
which our work is based, is that  in all the previous studied
cases, the function $w\mapsto G(w,t)$ is a semicomplete vector
field for all fixed $t\geq 0$. And in fact we prove that all
evolution families of the unit disc  are in one-to-one
correspondence with weak holomorphic vector fields which are
infinitesimal generators for almost every time (see Section 2
for definitions).

More precisely, we call {\sl Herglotz vector field of order
$d\geq 1$} a function $G:\D\times [0,+\infty)\to \C$ which is a
weak holomorphic vector field of order $d\geq 1$ (in the sense
of Carath\'eodory's theory, see Definition \ref{Definicion-VF})
and for almost every $t\geq 0$ has the property that $z\mapsto
G(z,t)$ is an infinitesimal generator. Also, an evolution
family of the unit disc is said to be {\sl of order $d\geq 1$}
if $|\varphi_{s,u}(z)-\varphi_{s,t}(z)|$ is locally bounded by
a non-negative function whose derivative is in $L^d$ (see
Definition \ref{def-ev})

Our main result is the following:

\begin{theorem}\label{main}
For any evolution family $(\v_{s,t})$ of order $d\geq 1$ in the
unit disc there exists a (essentially) unique  Herglotz vector
field $G(z,t)$ of order $d$ such that for all $z\in \D$
\begin{equation}\label{main-eq}
\frac{\de \v_{s,t}(z)}{\de t}=G(\v_{s,t}(z),t) \quad \hbox{a.e.
$t\in [0,+\infty)$}.
\end{equation}
Conversely, for any Herglotz vector field $G(z,t)$ of order
$d\geq 1$
 in the unit disc
there exists a unique evolution family $(\v_{s,t})$ of order
$d$  such that \eqref{main-eq} is satisfied.
\end{theorem}

Here {\sl essentially unique} means that if $H(z,t)$ is another
Herglotz vector field which satisfies \eqref{main-eq} then
$G(\cdot,t)=H(\cdot,t)$ for almost every $t\in [0,+\infty)$.

Infinitesimal generators have been characterized in several
different ways. In particular, in the proof of the above
theorem we use a result of \cite{Bracci-Contreras-Diaz-JEMS},
from which it follows that
$(d\rho_\D)_{(z,w)}(G(z,t),G(w,t))\leq 0$ for all $t\geq 0$ and
$z\neq w$, where $\rho_\D$ is the hyperbolic distance on $\D$.
This estimate allows us to avoid considering displacement of
fixed points in order to obtain suitable bounds. In fact, a
version of Theorem \ref{main} holds more generally on complex
complete hyperbolic manifolds whose Kobayashi distance is $C^1$
(see \cite{BCM2}).

In the unit disc we have a better description of Herglotz
vector fields, namely, a Berkson-Porta type formula holds for
non-autonomous vector fields which generate evolution families.
We say that a function $p: \D\times[0,+\infty)\to \C$ is a {\sl
Herglotz function of order $d\geq 1$} if it is locally in $L^d$
in $t\geq 0$, holomorphic in $z\in \D$  and $\Re p(z,t)\geq 0$
for all $z\in \D$ and $t\geq 0$ (see Definition
\ref{def-Her-fun}). The following representation formula holds:

\begin{theorem}\label{Berkson-Porta-for-Herglotz}
Let $G(z,t)$ be a Herglotz vector field  of order $d\geq 1$ in
the unit disc. Then there exist a (essentially) unique
measurable function  $\tau:[0,+\infty)\to \oD$  and a Herglotz
function $p(z,t)$ of order $d$ such that for all $z\in \D$
\begin{equation}\label{Herglotz-vf-main}
G(z,t)=(z-\tau(t))(\overline{\tau(t)}z-1)p(z,t) \quad
\hbox{a.e. $t\in [0,+\infty)$}.
\end{equation}
Conversely, given a  measurable function  $\tau:[0,+\infty)\to
\oD$  and a Herglotz function $p(z,t)$ of order $d\geq 1$,
equation \eqref{Herglotz-vf-main} defines a Herglotz vector
field of order $d$.
\end{theorem}

Here ``essentially unique'' means that $\tau, p$ are unique up
to changes on zero measure sets or on the set where $G\equiv 0$
(see  Theorem \ref{Herglotz-implica-VF} for a precise
statement).

There is thus an (essentially) one-to-one correspondence among
evolution families $(\v_{s,t})$ of order $d\geq 1$, Herglotz
vector fields $G(z,t)$ of order $d\geq 1$, and couples
$(p,\tau)$ of Herglotz functions $p(z,t)$ of order $d$ and
measurable functions $\tau:[0,+\infty)\to \oD$. In what follows
we say that the couple $(p,\tau)$ is the {\sl Berkson-Porta
data} for $(\v_{s,t})$.

Going back to Loewner equations, the   previous two theorems
can be combined saying that the following differential equation
\begin{equation}\label{generalLoew}
\left\{
\begin{array}
[c]{l}%
\overset{\bullet}{w}=(w-\tau(t))(\overline{\tau(t)}w-1)p(w,t)\quad\text{
for a. e.
}t\in\lbrack s,+\infty)\\
w(s)=z.
\end{array}
\right.
\end{equation}
has a family of solutions $(\v_{s,t})$ which form an evolution
family of order $d\geq 1$ provided $p(w,t)$ is a Herglotz
function of order $d\geq 1$ and $\tau:[0,+\infty)\to \oD$ is a
measurable function.

We point out that equation \eqref{generalLoew} contains all the
Loewner type equations studied so far in the literature, where
in fact only  Herglotz functions of order $\infty$ and
$\tau\equiv \hbox{const}$ are considered. In case $\tau\equiv
\hbox{const}$, evolution families of order $d\geq 1$ can be
defined by means of weaker conditions, such as properties of
regularity of first derivatives (see Theorem
\ref{EFwithcommonDW}).

The plan of the paper is the following. In Section 2 we collect
some preliminary results from iteration theory and semigroups
theory. In Section 3 we deal with evolution families, proving
some results about continuity in the two parameters. In Section
4 we introduce Herglotz vector fields and prove that they are
semicomplete (Theorem \ref{semicompletezza}). Then we relate
Herglotz vector fields with Herglotz functions (Theorem
\ref{Herglotz-implica-VF}) proving thus Theorem
\ref{Berkson-Porta-for-Herglotz}. In Section 5 we prove Theorem
\ref{Herglotz-implica-EF} which shows that solutions of a
Herglotz vector field form an evolution family (proving thus
one  part of Theorem \ref{main}). In Section 6 we prove the
other part of Theorem \ref{main} (see Theorem
\ref{EF-implica-VF}). With such a result at hand, moving from
evolution families to Herglotz vector fields and Herglotz
functions, we can prove some more regularity properties of
evolution families with respect to the two parameters (Theorem
\ref{Derivacion independiente de z (en la t)} and Theorem
\ref{Derivacion indep. de z (en la s)}). In particular, we show
that
\begin{equation*}
\frac{\partial\varphi}{\partial
s}(z,s,t)=-G(z,s)\varphi_{s,t}^{\prime}(z).
\end{equation*}
In Corollary \ref{Univalencia} we show that all the elements of
an evolution family must be univalent and, in Corollary
\ref{uniqueHerglotz}, that Herglotz vector fields are (almost
everywhere) characterized by their trajectories proving the
essential uniqueness of the previous Theorems \ref{main} and
\ref{Berkson-Porta-for-Herglotz}.

In the last two sections of the paper we get back to radial and
chordal Loewner equations. Namely, in Section 7 we turn our
attention to the case of a common fixed point (either in $\D$
or $\de \D$), proving regularity of the first derivative at the
common Denjoy-Wolff point (see Theorem
\ref{Continuidad-absoluta-multiplicador}). Finally, in Section
8 we concentrate on the case of a common fixed point on $\de
\D$, translating our results to the right half-plane and
including the previous cited results in our framework.

\medskip

We thank prof. Laszlo Lempert for a useful suggestion which
allowed us to prove directly Lemma \ref{laszlo}.

\section{Preliminaries from iteration theory}\label{preli}

As usual, we use the symbol $\angle$ before a limit to denote
the angular (non-tangential) limit either in the unit disc  or
in the right half-plane. For a given self-map $f$ of
${\mathbb{D}}$ and a point $p\in\partial{\mathbb{D}}$, we say
that $p$ is a (boundary) fixed point of the function $f$ if
$\angle \lim_{z\rightarrow p}f(z)=p.$ In general, if the
angular limit $q=\angle \lim_{z\rightarrow p}f(z)$ also belongs
to $\partial{\mathbb{D}}$, then the angular limit
$\angle\lim_{z\rightarrow p}\dfrac{f(z)-q}{z-p}$ exists (on the
Riemann sphere
$\widehat{{\mathbb{C}}}={\mathbb{C}}\cup\{\infty\}$) and it is
different from zero (see \cite{Pommerenke-II}). This limit is
known as the \textit{angular derivative\/} of $f$ at $p$ (in
the sense of Carath\'{e}odory) and we denote it by
$f^{\prime}(p)$.

We will write $f_{n}$ for the $n$-th iterate of a self-map $f$
of ${\mathbb{D}}$, defined inductively by $f_{1}=f$ and
$f_{n+1}=f\circ f_{n}$, $n\in{\mathbb{N}}$.

It can be  easily deduced from the Schwarz-Pick lemma that a
non-identity self-map $f$ of the unit disc  can have at most
one fixed point in ${\mathbb{D}}$. If such a unique fixed point
in ${\mathbb{D}}$ exists, it is usually called the
\textit{Denjoy-Wolff point\/}. The sequence of iterates
$\{f_{n}\}$ of $f$ converges to it uniformly on the compact
subsets of ${\mathbb{D}}$ whenever $f $ is not a disc
automorphism.

If $f$ has no fixed points in ${\mathbb{D}}$, the Denjoy-Wolff
theorem (see \cite{Abate}) guarantees the existence of a unique
point $\tau$ on the unit circle $\partial{\mathbb{D}}$ which is
the \textit{attractive fixed point\/}, that is, the sequence of
iterates $\{f_{n}\}$ converges to $\tau$ uniformly on the
compact subsets of ${\mathbb{D}}$. Such $\tau$ is again called
the \textit{Denjoy-Wolff point\/} of $f$. When
$\tau\in\partial{\mathbb{D}}$ is the Denjoy-Wolff point of $f$,
then $f^{\prime}(\tau)$ is actually real-valued and, moreover,
$0<f^{\prime}(\tau)\leq1$ (see \cite{Pommerenke-II}). Note that
$f$ can have other (boundary) fixed points. If
$p\in\partial\mathbb{D} $ is a fixed point of $f$ different
from the Denjoy-Wolff point then $f^{\prime
}(p)\in(1,+\infty)\cup\{\infty\}.$ As is often done in the
literature, we classify the holomorphic self-maps of the disc
into three categories according to their behavior near the
Denjoy-Wolff point:

\begin{itemize}
\item[(a)] \textit{elliptic\/}: the ones with a fixed point inside the disc ;

\item[(b)] \textit{hyperbolic\/}: the ones with the Denjoy-Wolff point
$\tau\in\partial{\mathbb{D}}$ such that $f^{\prime}(\tau)<1$;

\item[(c)] \textit{parabolic\/}: the ones with the Denjoy-Wolff point $\tau
\in\partial{\mathbb{D}}$ such that $f^{\prime}(\tau)=1$.
\end{itemize}

The following simple and standard procedure is suitable for
both hyperbolic and parabolic maps. Let $\tau$ be the
Denjoy-Wolff point of a self-map $f$ of ${\mathbb{D}}$, with
$|\tau|=1$. The Cayley transform
$T_{\tau}(z)=\frac{\tau+z}{\tau-z}$ maps ${\mathbb{D}}$
conformally onto the right half-plane
${\mathbb{H}}=\{z\,\colon\,\text{Re\thinspace}z>0\}$ and takes
the point $\tau$ to infinity. Thus, to every self-map $f$ of
${\mathbb{D}}$ there corresponds a unique self-map $g$ of
${\mathbb{H}}$, called the \textit{conjugate map\/} of $f$,
such that $g=T_{\tau}\circ f\circ T_{\tau}^{-1}$ with
Denjoy-Wolff point at $\infty$ in ${\mathbb{H}}$. Namely,
$\angle\lim_{w\rightarrow\infty}\frac
{g(w)}{w}=f^{\prime}(\tau)^{-1}.$

A \textit{(one-parameter) semigroup of holomorphic functions}
is a  continuous homomorphism $\Phi:t\mapsto\Phi(t)=\phi_{t}$
from the additive semigroup of non-negative real numbers into
the composition semigroup of holomorphic self-maps of
$\mathbb{D}$. Namely, $\Phi$ satisfies the following three
conditions:

\begin{enumerate}
\item[S1.] $\phi_{0}$ is the identity in $\mathbb{D},$

\item[S2.] $\phi_{t+s}=\phi_{t}\circ\phi_{s},$ for all $t,s\geq0,$

\item[S3.] $\phi_{t}(z)$ tends to $z$ as $t$ tends to $0,$ uniformly on
compact subsets of $\mathbb{D}.$
\end{enumerate}

Given a semigroup $\Phi=(\phi_{t})$, it is well-known (see
\cite{Shoikhet}, \cite{Berkson-Porta}) that there exists a
\textit{unique} holomorphic function $G:\mathbb{D\rightarrow
C}$ such that,
\[
\frac{\partial\phi_{t}(z)}{\partial t}=G\left(  \phi_{t}(z)\right)  =G\left(
z\right)  \frac{\partial\phi_{t}(z)}{\partial z}\quad\text{for all }%
z\in\mathbb{D}\text{ and }t\geq0.
\]
To simplify the notation, we denote
$\phi_{t}^{\prime}(z)=\dfrac{\partial \phi_{t}(z)}{\partial
z}.$ In what follows, $G$ will be called the \textit{vector
field} associated with $\Phi\ $or the (\textit{infinitesimal)
generator }of $\Phi.$ We warn the reader that, in
\cite{Shoikhet},\ these vector fields are introduced with a
different sign convention.

There is a very nice representation, due to Berkson and Porta
\cite{Berkson-Porta}, of those holomorphic functions on the
disc which are infinitesimal generators. A holomorphic function
$G:\mathbb{D\rightarrow C}$ is the infinitesimal generator of a
semigroup $\Phi$ of holomorphic self-maps of $\mathbb D$  if
and only if there exist  $\tau\in\overline{\mathbb{D}}$ and a
holomorphic function $p:\mathbb{D\rightarrow C}$ with $\Re
p\geq0$ such that
\[
G(z)=(\tau-z)(1-\overline{\tau}z)p(z),\text{ \quad}z\in\mathbb{D}.
\]
Moreover, if $G$ is not identically zero, then such a representation is
unique. In fact, the point $\tau$ is the \textit{Denjoy-Wolff point} of all
the functions of the semigroup.

We denote by $\mathrm{Gen}(\mathbb{D})$ the set of all the
infinitesimal generators of semigroups of holomorphic self-maps
of the unit disc. It is well-known that
$\mathrm{Gen}(\mathbb{D})$ is closed in
$\mathrm{Hol}(\mathbb{D},\mathbb{C})$ and a real convex cone in
$\mathrm{Hol}(\mathbb{D},\mathbb{C})$ with vertex at $0$ (see,
for example, \cite{Abate} and \cite{Shoikhet}). A useful
example of infinitesimal generator is given by $G=\varphi-\id$
for  $\varphi \in\mathrm{Hol}(\mathbb{D},\mathbb{D})$
\cite[Corollary 3.3.1]{Shoikhet}.

The following two facts, related to the continuity of the
so-called Heins map (see \cite{He} and \cite{Br1}) might be
known but, since we do not have a reference, we include their
proofs here for the sake of completeness.

\begin{proposition}
\label{BP-continuidad} Endow $\mathrm{Gen}(\mathbb{D})$ and
$\mathrm{Hol}(\mathbb{D},\mathbb{C})$  with the compact-open
topology, and let $\underline{0}$ denote the zero function,
$\underline{0}(z)=0$ for all $z\in \D$. For all $F\in
\mathrm{Gen}(\mathbb{D})$, using the Berkson-Porta
representation we write
\[
F(z)=(z-\tau_{F})(\overline{\tau_{F}}z-1)p_{F}(z),\text{
}z\in\mathbb{D}.
\]
Then the following two maps are continuous
\begin{align*}
BP_{\tau}  &
:\mathrm{Gen}(\mathbb{D})\setminus\{\underline{0}\}\rightarrow\overline
{\mathbb{D}},\qquad
\mathrm{Gen}(\mathbb{D})\setminus\{\underline{0}\}\ni F\mapsto
BP_{\tau
}(F):=\tau_{F}\\
BP_{p}  &
:\mathrm{Gen}(\mathbb{D})\setminus\{\underline{0}\}\rightarrow\mathrm{Hol}
(\mathbb{D},\mathbb{C}),\qquad
\mathrm{Gen}(\mathbb{D})\setminus\{\underline{0}\}\ni F\mapsto
BP_{p}(F):=p_{F}.
\end{align*}
\end{proposition}

\begin{proof}
By the uniqueness of the Berkson-Porta representation, the maps
$BP_{\tau}$ and $BP_{p}$ are well-defined. We only give the
proof for the continuity of $BP_{\tau}$, because the other is
almost identical. Let $\{F_{n}\}\subset
\mathrm{Gen}(\mathbb{D})\setminus\{\underline{0}\}$ converging
to  $F\in \mathrm{Gen}(\mathbb{D})\setminus\{\underline{0}\}$.
Let $\tau _{n}:=BP_{\tau}(F_{n})$ and $\tau:=BP_{\tau}(F)$. We
need to show that $\tau_n\to \tau$. To this aim, it is enough
to show that any converging subsequence of $\{\tau _{n}\}$
converges to $\tau$. Let $\{\tau_{n_{k}}\}$ be a  subsequence
converging to some $\alpha\in\overline{\mathbb{D}}$. Since
$\{w\in \C: \Re w>0\}$ is hyperbolic, the family
$\{p_{n}:=BP_{p}(F_{n}):n\in\mathbb{N}\}$ is a normal family in
$\mathrm{Hol}(\mathbb{D},\mathbb{C})$. The sequence $\{F_{n}\}$
is  convergent and thus, up to extract subsequences, we can
assume that $\tau_{n_{k}}\to\alpha$ and $p_{n_{k}}\to p$ for
some $p\in\mathrm{Hol}(\mathbb{D},\mathbb{C})$ with $\Re
p\geq0$.
Therefore, for all $z\in\mathbb{D},$%
\[
F(z)=\lim_{k\to \infty
}F_{n_{k}}(z)=(z-\alpha)(\overline{\alpha}z-1)p(z).
\]
On the other hand, $F(z)=$
$(z-\tau)(\overline{\tau}z-1)p_{F}(z)$. By the uniqueness of
the Berkson-Porta representation, we conclude that
$\alpha=\tau$ as wanted.
\end{proof}

\begin{lemma}
\label{Lema de los dos puntos} Let $\{G_{n}\}$ be a sequence in $\mathrm{Gen}(\mathbb{D}%
)$ such that there are two different points
$z_{0},z_{1}\in\mathbb{D}$ and two sequences $\{u_{n}\}$ and
$\{v_{n}\}$ in $\mathbb{D}$ with $\lim_{n}u_{n}=z_{0}$ and
$\lim_{n}v_{n}=z_{1}$ such that
\[
\sup_{n}|G_{n}(u_{n})|<+\infty\text{ \quad and }\quad\sup_{n}|G_{n}%
(v_{n})|<+\infty.
\]
Then there exists a subsequence $\{G_{n_{k}}\}$ converging to
an infinitesimal generator $G\in\mathrm{Gen}(\mathbb{D})$.
\end{lemma}

\begin{proof}
By Berkson-Porta's theorem, there are points $\tau_{n}\in\overline{\mathbb{D}%
}$ and holomorphic maps
$p_{n}:\mathbb{D}\rightarrow\mathbb{C},$ with $\Re p_{n}\geq0$,
such that $G_{n}(z)=(z-\tau_{n})(\overline
{\tau_{n}}z-1)p_{n}(z)$ for all $z\in\mathbb{D}$. Since the
sequence $\{\tau_{n}\}$ is bounded and
$\{p_{n}:n\in\mathbb{N}\}$ is a normal family, there exist a
strictly increasing sequence of natural numbers $\{n_{k}\}$ and
a point $\tau\in\overline{\mathbb{D}}$ such that
$\tau_{n_{k}}\rightarrow\tau$ and $p_{n_{k}}$ converges
uniformly on compacta  either to an holomorphic function
$p:\mathbb{D}\rightarrow\mathbb{C}$ or to $\infty.$

Suppose that $\{p_{n_{k}}\}$  compactly diverges
  to $\infty.$ Since $z_{0}$ and
$z_{1}$ are different, we may assume that
$\tau\neq z_{0}$. Then we have that%
\begin{align*}
+\infty & >\sup_{n}|G_{n}(u_{n})|\geq\lim_{k}|G_{n_{k}}(u_{n_{k}})|=\lim
_{k}|(u_{n_{k}}-\tau_{n_{k}})(\overline{\tau_{n_{k}}}u_{n_{k}}-1)p_{n_{k}%
}(u_{n_{k}})|\\
& =|(z_{0}-\tau)(\overline{\tau}z_{0}-1)|\lim_{k}|p_{n_{k}}(u_{n_{k}%
})|=+\infty.
\end{align*}
A contradiction. So $(p_{n_{k}})$ converges uniformly on
compacta to a holomorphic function
$p:\mathbb{D}\rightarrow\mathbb{C}$ with $\Re p\geq0.$ Letting
$G(z)=(z-\tau)(\overline{\tau}z-1)p(z)$, it follows then that
$(G_{n_{k}})$ converges uniformly on compacta to $G$ and, again
by Berkson-Porta's theorem, $G$ is an infinitesimal generator.
\end{proof}

\begin{remark}
It is worth noticing that the above lemma would not be  true if
we only assume that there is only one point
$z_{0}\in\mathbb{D}$ and one sequence $(u_{n}) $ in
$\mathbb{D}$ with $\lim_{n}u_{n}=z_{0}$ such that
\[
\sup_{n}|G_{n}(u_{n})|<+\infty\text{.}%
\]
For example,  consider the sequence of infinitesimal generators
given by $G_{n}(z):=-nz$, for all $z\in\mathbb{D}$ ($G_{n}$ is
the infinitesimal generator of the semigroup
$\varphi_{t}(z)=e^{-nt}z$).
\end{remark}

\section{Evolution families in the unit disc}

\begin{definition}\label{def-ev}
A family $(\varphi_{s,t})_{0\leq s\leq t<+\infty}$ of
holomorphic self-maps of the unit disc  is an {\sl evolution
family of order $d$} with $d\in [1,+\infty]$ (in short, an {\sl
$L^d$-evolution family}) if

\begin{enumerate}
\item[EF1.] $\varphi_{s,s}=id_{\mathbb{D}},$

\item[EF2.] $\varphi_{s,t}=\varphi_{u,t}\circ\varphi_{s,u}$ for all $0\leq
s\leq u\leq t<+\infty,$

\item[EF3.] for all $z\in\mathbb{D}$ and for all $T>0$ there exists a
non-negative function $k_{z,T}\in L^{d}([0,T],\mathbb{R})$ such
that
\[
|\varphi_{s,u}(z)-\varphi_{s,t}(z)|\leq\int_{u}^{t}k_{z,T}(\xi)d\xi
\]
for all $0\leq s\leq u\leq t\leq T.$
\end{enumerate}
\end{definition}

Sometimes in the proofs, and for the sake of clearness, we will use the
notation $\varphi(z,s,t)$ instead of $\varphi_{s,t}(z),$ where $z\in
\mathbb{D}$ and $0\leq s\leq t.$

\begin{remark}
Clearly, by the very definition, if $(\varphi_{s,t})_{0\leq
s\leq t<+\infty}$  is an evolution family of order $d$ then it
is also an evolution family of order $d'$ for all $1\leq d'\leq
d$.
\end{remark}

\begin{example}
Let $d\geq 1$. Let $\lambda:[0,+\infty)\to \R^+$ be an
absolutely continuous increasing function such that
$\overset{\bullet}{\lambda}\in L^d_{{\sf loc}}([0,+\infty),\R)$
but  $\overset{\bullet}{\lambda}\not\in L^{k}_{{\sf
loc}}([0,+\infty),\R)$ for any $k>d$. Then $\v_{s,t}(z):=\exp(
\lambda(s)-\lambda(t)) z$ is an evolution family of order $d$
which is not of order $k$ for any $k>d$.
\end{example}

\begin{example}
Let $(\phi_{t})$ be a semigroup of holomorphic self-maps of
$\D$. Let $\varphi_{s,t}:=\phi_{t-s}$ for  $0\leq s\leq
t<+\infty.$ Then $(\varphi_{s,t})$ is an evolution family of
order $\infty.$ Indeed, clearly the family $(\varphi_{s,t}) $
satisfies EF1 and EF2. We have only have to check  EF3. Fix
$z\in\mathbb{D}$ and $T>0.$ Then there is a number $R$ such
that $|\phi_{\xi}(z)|\leq R$ for all $0\leq\xi\leq T.$
Therefore, there is $M=M(z,T)>0$ such that
$|G(\phi_{\xi}(z))|\leq M,$ for all $0\leq\xi\leq T$, where $G$
is the infinitesimal generator of the
semigroup. Then, for all $0\leq s\leq u\leq t\leq T,$ we have%
\begin{equation*}
\begin{split}
|\varphi_{s,u}(z)-\varphi_{s,t}(z)|&\leq|\phi_{u-s}(z)-\phi_{t-s}%
(z)|=\left\vert \int_{u-s}^{t-s}\frac{\partial\phi_{\xi}(z)}{\partial\xi}%
d\xi\right\vert \\&=\left\vert
\int_{u-s}^{t-s}G(\phi_{\xi}(z))d\xi\right\vert
\leq\int_{u}^{t}Md\xi.
\end{split}
\end{equation*}

\end{example}

In this preliminary section we state some properties related to
continuity of the evolution family with respect to the real
parameters. Throughout the paper, we denote by
$\rho_{\mathbb{D}}(z,w)$ the hyperbolic distance in the unit
disc between two points $z,w\in\mathbb{D}$.

\begin{proposition}
\label{EF-continuidad} Let $d\geq 1$ and let $(\varphi_{s,t})$
be an evolution family of order $d$ in the unit disc. The map
$(s,t)\mapsto\varphi_{s,t}\in\mathrm{Hol}(\mathbb{D}
,\mathbb{C})$ is jointly continuous. Namely, given a compact
set $K\subset\mathbb{D}$ and two sequences $\{s_{n}\},$
$\{t_{n}\}$ in $[0,+\infty),$ with $0\leq s_{n}\leq t_{n},$
$s_{n}\rightarrow s,$ and $t_{n}\rightarrow t$, then
$\lim_{n\to \infty}\varphi_{s_{n},t_{n}}=\varphi_{s,t}$
uniformly on $K.$
\end{proposition}

\begin{proof}
Let $\{s_{n}\},$ $\{t_{n}\}$ be two sequences in $[0,+\infty)$
with $0\leq s_{n}\leq t_{n},$ $s_{n}\rightarrow s$ and
$t_{n}\rightarrow t$. Since the
set $\{\varphi_{u,v}:0\leq u\leq v\}$ is bounded in $\mathrm{Hol}%
(\mathbb{D},\mathbb{C}),$ by Vitali's theorem, it is enough to
show that $\lim_{n\to \infty
}\varphi_{s_{n},t_{n}}(z)=\varphi_{s,t}(z)$ for all fixed $z$
in the unit disc. Fix a point $z\in\mathbb{D}$. In order to
obtain the result, we may (and we do) assume that the sequences
$\{s_{n}\}$ and $\{t_{n}\}$ are in one of the following three
cases:

Case I:  $s_{n}\leq t_{n}\leq s$ for all $n.$

Case II: $s\leq s_{n}$ for all $n;$

Case III: $s_{n}\leq s\leq t_{n}$ for all $n;$

In case I, we have that $s=t$. Therefore, using EF3, we take
the corresponding function $k_{z,t}\in L^{d}([0,T],\mathbb{R})$
and
\begin{align*}
\left\vert \varphi_{s_{n},t_{n}}(z)-\varphi_{s,t}(z)\right\vert
& =\left\vert
\varphi_{s_{n},t_{n}}(z)-z\right\vert =\left\vert \varphi_{s_{n},t_{n}%
}(z)-\varphi_{s_{n},s_{n}}(z)\right\vert \\
&
\leq\int_{s_{n}}^{t_{n}}k_{z,t}(\xi)d\xi\overset{n}{\longrightarrow}0,
\end{align*}
where the last limit is zero because the measure of the
interval $[s_{n},t_{n}]$ tends to zero as $n$ goes to $\infty.$

If we are in case II, then
\begin{align*}
\rho_{\mathbb{D}}(\varphi_{s_{n},t_{n}}(z),\varphi_{s,t}(z))  & \leq
\rho_{\mathbb{D}}(\varphi_{s_{n},t_{n}}(z),\varphi_{s,t_{n}}(z))+\rho
_{\mathbb{D}}(\varphi_{s,t_{n}}(z),\varphi_{s,t}(z))\\
& =\rho_{\mathbb{D}}(\varphi_{s_{n},t_{n}}(z),\varphi_{s_{n},t_{n}}%
(\varphi_{s,s_{n}}(z)))+\rho_{\mathbb{D}}(\varphi_{s,t_{n}}(z),\varphi
_{s,t}(z))\\
& \leq\rho_{\mathbb{D}}(z,\varphi_{s,s_{n}}(z))+\rho_{\mathbb{D}}%
(\varphi_{s,t_{n}}(z),\varphi_{s,t}(z)),
\end{align*}
while in case III,
\begin{align*}
\rho_{\mathbb{D}}(\varphi_{s_{n},t_{n}}(z),\varphi_{s,t}(z))  & \leq
\rho_{\mathbb{D}}(\varphi_{s_{n},t_{n}}(z),\varphi_{s,t_{n}}(z))+\rho
_{\mathbb{D}}(\varphi_{s,t_{n}}(z),\varphi_{s,t}(z))\\
& =\rho_{\mathbb{D}}(\varphi_{s,t_{n}}(\varphi_{s_{n},s}(z)),\varphi_{s,t_{n}%
}(z))+\rho_{\mathbb{D}}(\varphi_{s,t_{n}}(z),\varphi_{s,t}(z))\\
& \leq\rho_{\mathbb{D}}(\varphi_{s_{n},s}(z),z)+\rho_{\mathbb{D}}%
(\varphi_{s,t_{n}}(z),\varphi_{s,t}(z)).
\end{align*}

Therefore, bearing in mind that $\varphi_{s,t}(z)$ and $z$
belong to $\mathbb{D}$, to end up the proof it is enough to
show that the sequence $\{\varphi_{s,t_{n}}(z)\}$ converges to
$\varphi_{s,t}(z)$ and $\{\varphi_{s_{n},s}(z)\}$ (or
$\{\varphi_{s,s_{n}}(z)$\}) converges to $z$ as $n$ goes to
$+\infty.$

Let $T>\sup_{n}t_{n}.$ By the very definition of evolution
family, there exists a non-negative function $k_{z,T}\in
L^{d}([0,T],\mathbb{R})$ such that
\[
|\varphi_{r,u}(z)-\varphi_{r,v}(z)|\leq\int_{u}^{v}k_{z,T}(\xi)d\xi
\]
for all $0\leq r\leq u\leq v\leq T.$

Since $k_{z,T}\in L^{1}([0,T],\mathbb{R}),$ we have
\[
|\varphi_{s,t_{n}}(z)-\varphi_{s,t}(z)|\leq\int_{t}^{t_{n}}k_{z,T}(\xi
)d\xi\underset{n\rightarrow\infty}{\longrightarrow}0,
\]
if $s\leq s_{n}$ for all $n,$%
\[
|\varphi_{s,s_{n}}(z)-z|=|\varphi_{s,s_{n}}(z)-\varphi_{s,s}(z)|\leq\int
_{s}^{s_{n}}k_{z,T}(\xi)d\xi\underset{n\rightarrow\infty}{\longrightarrow}0,
\]
and, if $s_{n}\leq s$ for all $n,$
\[
|\varphi_{s_{n},s}(z)-z|=|\varphi_{s_{n},s}(z)-\varphi_{s_n,s_n}(z)|\leq
\int_{s_{n}}^{s}k_{z,T}(\xi)d\xi\underset{n\rightarrow\infty}{\longrightarrow
}0.
\]

\end{proof}

\begin{lemma}
\label{EF-acotacion} Let $(\varphi_{s,t})$ be an evolution
family of order $d\geq 1$ in the unit disc  $\mathbb{D}.$ Then
for each $0<T<+\infty$ and $0<r<1,$ there exists $R=R(r,T)<1$
such that
\[
|\varphi_{s,t}(z)|\leq R
\]
for all $0\leq s\leq t\leq T$ and $|z|\leq r.$
\end{lemma}

\begin{proof}
Suppose that the lemma is not true. Then there exist three sequences
$\{z_{n}\},$ $\{s_{n}\},$ and $\{t_{n}\}$ such that $|z_{n}|\leq r,$ $z_{n}%
\rightarrow z_{0},$ $s_{n},t_{n}\in\lbrack0,T],$ $s_{n}\leq t_{n},$
$s_{n}\rightarrow s_{0},$ $t_{n}\rightarrow t_{0},$ and $|\varphi_{s_{n}%
,t_{n}}(z_{n})|\rightarrow1.$ Since the map $\varphi_{s_{n},t_{n}}$ is a
contraction for the hyperbolic metric, we have that $\rho_{\mathbb{D}}%
(\varphi_{s_{n},t_{n}}(z_{n}),\varphi_{s_{n},t_{n}}(z_{0}))\leq\rho
_{\mathbb{D}}(z_{n},z_{0})\rightarrow0.$ Then $|\varphi_{s_{n},t_{n}}%
(z_{0})|\rightarrow1.$ By Proposition \ref{BP-continuidad}, the map
$t\mapsto\varphi_{0,t}(z_{0})$ is continuous. Moreover,
\begin{align*}
\rho_{\mathbb{D}}(\varphi_{0,t_{n}}(z_{0}),\varphi_{s_{n},t_{n}}(z_{0}))  &
=\rho_{\mathbb{D}}(\varphi_{s_{n},t_{n}}(\varphi_{0,s_{n}}(z_{0}%
)),\varphi_{s_{n},t_{n}}(z_{0}))\\
& \leq\rho_{\mathbb{D}}(\varphi_{0,s_{n}}(z_{0}),z_{0})\rightarrow
\rho_{\mathbb{D}}(\varphi_{0,s}(z_{0}),z_{0})<+\infty.
\end{align*}
Again this implies that $|\varphi_{0,t_{n}}(z_{0})|\rightarrow1.$ But
$\varphi_{0,t_{n}}(z_{0})\rightarrow\varphi_{0,t}(z_{0})\in\mathbb{D}$. A contradiction.
\end{proof}

\begin{proposition}
\label{EF-continuidadabsoltuta} Let $(\varphi_{s,t})$ be an
evolution family of order $d\geq 1$ in the unit disc
$\mathbb{D}.$

\begin{enumerate}
\item For all $z\in\mathbb{D}$ and for all $s\geq0,$ the map $\lbrack
s,\infty)\ni t\mapsto\varphi_{s,t}(z)\in\mathbb{C}$ is locally
absolutely continuous, that is, for all $T>s,$ the map $\lbrack
s,T]\ni t\mapsto \varphi_{s,t}(z)\in\mathbb{C}$ is absolutely
continuous$.$

\item For all $z\in\mathbb{D}$ and for all $T>0,$ the map $\lbrack
0,T]\ni s\mapsto\varphi_{s,T}(z)\in\mathbb{C}$ is absolutely
continuous$.$
\end{enumerate}
\end{proposition}

\begin{proof}
(1) Let us fix $z\in\mathbb{D}$ and two non-negative numbers
$0\leq s<T.$ Then there exists a non-negative function
$k_{z,T}\in L^{d}([0,T],\mathbb{R}) $ such that
\[
|\varphi_{s,u}(z)-\varphi_{s,t}(z)|\leq\int_{u}^{t}k_{z,T}(\xi)d\xi
\]
for all $0\leq s\leq u\leq t\leq T.$ Since $k_{z,T}\in
L^{d}([0,T],\mathbb{R}),$ the map $\lbrack 0,T]\ni
t\mapsto\int_{0}^{t}k_{z,T}(\xi)d\xi$ is absolutely continuous
and this clearly implies that the map $\lbrack s,T]\ni
t\mapsto\varphi _{s,t}(z)$ is absolutely continuous.

(2) Fix $z\in\mathbb{D}$ and $T>0.$ By Lemma \ref{EF-acotacion}, there is
$R=R(z,T)<1$ such that
\[
|\varphi_{s,t}(z)|\leq R
\]
for all $0\leq s\leq t\leq T.$ Take $r=(R+1)/2.$

By Cauchy integral's formula, if $0\leq a<b\leq T,$ we have
\begin{align*}
|\varphi_{b,T}(z)-\varphi_{a,T}(z)|  & =|\varphi_{b,T}(z)-\varphi
_{b,T}(\varphi_{a,b}(z))|\\
& =\left\vert \frac{1}{2\pi i}%
{\displaystyle\int\nolimits_{C(0,r)^{+}}}
\frac{\varphi_{b,T}(\xi)}{\xi-z}d\xi-\frac{1}{2\pi i}%
{\displaystyle\int\nolimits_{C(0,r)^{+}}}
\frac{\varphi_{b,T}(\xi)}{\xi-\varphi_{a,b}(z)}d\xi\right\vert \\
& =\left\vert \frac{1}{2\pi}%
{\displaystyle\int\nolimits_{C(0,r)^{+}}}
\varphi_{b,T}(\xi)\frac{z-\varphi_{a,b}(z)}{(\xi-z)(\xi-\varphi_{a,b}(z))}%
d\xi\right\vert \\
& \leq r\frac{|z-\varphi_{a,b}(z)|}{(r-|z|)(r-R)}\leq\frac{4}{(1-R)^{2}%
}|z-\varphi_{a,b}(z)|.
\end{align*}
Now, let $k_{z,T}\in L^{d}([0,T],\mathbb{R})$ be a non-negative
function such that
\[
|\varphi_{s,u}(z)-\varphi_{s,t}(z)|\leq\int_{u}^{t}k_{z,T}(\xi)d\xi
\]
for all $0\leq s\leq u\leq t\leq T.$ We have
\[
|\varphi_{b,T}(z)-\varphi_{a,T}(z)|\leq\frac{4}{(1-R)^{2}}|\varphi
_{a,a}(z)-\varphi_{a,b}(z)|\leq\frac{4}{(1-R)^{2}}\int_{a}^{b}k_{z,T}(\xi
)d\xi.
\]
Again, since $k_{z,T}\in L^{1}([0,T],\mathbb{R}),$ this implies
that the map $s\in\lbrack0,T]\mapsto\varphi_{s,T}(z)$ is
absolutely continuous.
\end{proof}

\section{Weak holomorphic vector fields and Herglotz vector fields}

\begin{definition}
\label{Definicion-VF} Let $d\in [1,+\infty]$. A {\sl weak
holomorphic vector field of order $d$} on the unit disc
$\mathbb{D}$ is a function
$G:\mathbb{D}\times\lbrack0,+\infty)\rightarrow \mathbb{C}$
with the following properties:

\begin{enumerate}
\item[WHVF1.] For all $z\in\mathbb{D},$ the function $\lbrack
0,+\infty)\ni t\mapsto G(z,t)$ is measurable;

\item[WHVF2.] For all $t\in\lbrack0,+\infty),$ the function $
\mathbb{D}\ni z\mapsto G(z,t)$ is holomorphic;

\item[WHVF3.] For any compact set $K\subset\mathbb{D}$ and for all $T>0$ there
exists a non-negative function $k_{K,T}\in
L^{d}([0,T],\mathbb{R})$ such that
\[
|G(z,t)|\leq k_{K,T}(t)
\]
for all $z\in K$ and for almost every $t\in\lbrack0,T].$
\end{enumerate}
\end{definition}

\begin{lemma}
\label{WHVF4} Let $G$ be a weak holomorphic vector field of
order $d\geq 1$ on the unit disc $\mathbb{D}$. Then for any
compact set $K\subset\mathbb{D}$ and for all $T>0,
$ there exists a non-negative function $\widehat{k}_{K,T}\in L^{d}%
([0,T],\mathbb{R})$ such that
\[
|G(z,t)-G(w,t)|\leq\widehat{k}_{K,T}(t)|z-w|
\]
for all $z,w\in K$ and for almost every $t\in\lbrack0,T].$
\end{lemma}

\begin{proof}
Fix a compact set $K\subset\mathbb{D}$ and $T>0.$ Take $0<r<1$
such that $K\subset \D(0,r):=\{\zeta\in\C: |\zeta|<r\}.$ Let
$A:=\overline{\D(0,(r+1)/2)}.$ By the very definition of weak
holomorphic vector field, there exists a non-negative function
$k_{A,T}\in L^{d}([0,T],\mathbb{R})$ such that
\[
|G(z,t)|\leq k_{A,T}(t)
\]
for all $z\in A$ and for almost every $t\in\lbrack0,T].$ Since
the function $\mathbb{D}\ni z\mapsto G(z,t)$ is holomorphic,
taking $z,w\in K,$ we have
\begin{align*}
|G(z,t)-G(w,t)|  & =\left\vert \frac{1}{2\pi i}%
{\displaystyle\int\nolimits_{C(0,(r+1)/2)^{+}}}
\frac{G(\xi,t)}{\xi-z}d\xi-\frac{1}{2\pi i}%
{\displaystyle\int\nolimits_{C(0,(r+1)/2)^{+}}}
\frac{G(\xi,t)}{\xi-w}d\xi\right\vert \\
& =\frac{1}{2\pi}\left\vert
{\displaystyle\int\nolimits_{C(0,(r+1)/2)^{+}}}
G(\xi,t)\frac{z-w}{(\xi-z)(\xi-w)}d\xi\right\vert \\
& \leq\frac{1}{2\pi}%
{\displaystyle\int\nolimits_{C(0,(r+1)/2)^+}}
\left\vert G(\xi,t)\right\vert \frac{|z-w|}{|(\xi-z)(\xi-w)|}|d\xi|\\
& \leq\frac{1}{2\pi}%
{\displaystyle\int\nolimits_{C(0,(r+1)/2)^+}}
k_{A,T}(t)\frac{4|z-w|}{(1-r)^{2}}|d\xi|\leq4\frac{k_{A,T}(t)}{(1-r)^{2}%
}|z-w|.
\end{align*}
Thus, the result follows by choosing
$\widehat{k}_{K,T}:=4\frac{k_{A,T}}{(1-r)^{2}}.$
\end{proof}

By the Carath\'{e}odory theory of ODE's (see, for example,
\cite{Coddington-Levison}), it follows from the above lemma that if $G$ is a
weak holomorphic vector field on $\mathbb{D},$ then for any $(z,s)\in
\mathbb{D}\times\lbrack0,+\infty),$ there exist a unique $I(z,s)>s$ and a
function $x:[s,I(z,s))\rightarrow\mathbb{D}$ such that

\begin{enumerate}
\item $x$ is locally absolutely continuous in $[s,I(z,s))$, that is, $x$ is
absolutely continuous in $[s,T]$ for all $s<T<I(z,s);$

\item $x$ is the solution to the following problem:%
\[
\left\{
\begin{array}
[c]{l}%
\overset{\bullet}{x}(t)=G(x(t),t)\\
x(s)=z
\end{array}
\right.
\]
for almost all $t\in\lbrack s,I(z,s))$ with respect to the
Lebesgue measure.

\item The interval $[s,I(z,s))$ is maximal. Namely, if
 $y:[s,I)\rightarrow \D$ is a locally absolutely continuous function satisfying
\[
\left\{
\begin{array}
[c]{l}%
\overset{\bullet}{y}(t)=G(y(t),t)\\
y(s)=z
\end{array}
\right.
\]
for almost all $t\in\lbrack s,I)$, then $I\leq I(z,s)$ and
$x(t)=y(t)$ on $[s,I).$
\end{enumerate}

Such a map $x$ is known as the \textit{positive trajectory} of
the vector field $G$ at the pair $(z,s).$ The number $I(z,s)$
is known as the \textit{escaping time} for the couple $(z,s).$
We say that the weak holomorphic vector field is
\textit{semi-complete} if $I(z,s)=+\infty$ for all
$(z,s)\in\mathbb{D}\times\lbrack0,+\infty).$

\begin{definition}
Let $G(z,t)$ be a weak holomorphic vector field   of order
$d\in [1,+\infty]$ on the unit disc $\mathbb{D}$. We say that
$G$ is a {\sl (generalized) Herglotz vector field} (of order
$d$) if for almost every $t\in [0,+\infty)$ it follows
$G(\cdot, t)\in \mathrm{Gen}(\mathbb{D})$.
\end{definition}

Herglotz vector fields are always semicomplete:

\begin{theorem}\label{semicompletezza}
Let $G:\D\times[0,+\infty)\to \C$ be a Herglotz vector field of
order $d\in [1,+\infty)$. Then $G$ is semicomplete.
\end{theorem}

\begin{proof}
Denote by $\phi_{s,z}$ the positive trajectory associated with
the Cauchy problem
\[
\left\{
\begin{array}
[c]{l}%
\overset{\bullet}{x}(t)=G(x(t),t)\text{ }\\
x(s)=z,
\end{array}
\right.  \text{ }%
\]
and let $I(z,s)$ be the corresponding escaping time. We have to
show that
\begin{equation}\label{forever}
I(z,s)=+\infty \quad \hbox{for all $z\in \D$ and $s\geq 0$}.
\end{equation}

In order to prove \eqref{forever}, we first show that for all
$z,w\in\mathbb{D}$ and  $0\leq s\leq t<\min \{I(z,s),I(w,s)\}$
\begin{equation}\label{claim}
\rho_{\mathbb{D}}(\phi_{s,z}(t),\phi_{s,w}(t))\leq\rho_{\mathbb{D}}(z,w).
\end{equation}
Assume that $I(z,s)\leq I(w,s)$ and consider the function
$h(t):=\rho_{\mathbb{D}}(\phi_{s,z}(t),\phi_{s,w}(t))$ defined
for  $t\in\lbrack s,I(z,s))$. By Caratheodory's ODE's theory,
such a function is absolutely continuous and differentiable
almost everywhere.

On the other hand, by definition of Herglotz vector field,
fixed a point $t$,  the map $\mathbb{D}\ni z\mapsto G(z,t)$ is
the infinitesimal generator of a semigroup of holomorphic
self-maps of the unit disc for almost every $t\in [0,+\infty)$.
Therefore, by \cite[Thm. 0.2] {Bracci-Contreras-Diaz-JEMS} for
almost every $t\in [0,+\infty)$
\[
\left(  d\rho_{\mathbb{D}}\right)  _{\left(  \phi_{s,z}(t),\phi_{s,w}%
(t)\right)  }\left( G(\phi_{s,z}(t),t),G(\phi
_{s,w}(t),t)\right) \leq0.
\]
Hence, for almost every $t\geq 0$,
\begin{align*}
\overset{\bullet}{h}(t)  & =\left(  d\rho_{\mathbb{D}}\right)
_{\left(
\phi_{s,z}(t),\phi_{s,w}(t)\right)  }\left(  \overset{\bullet}{\phi_{s,z}%
}(t),\overset{\bullet}{\phi_{s,w}}(t)\right) \\
& =\left(  d\rho_{\mathbb{D}}\right)  _{\left(
\phi_{s,z}(t),\phi _{s,w}(t)\right)  }\left(
G(\phi_{s,z}(t),t),G(\phi _{s,w}(t),t)\right) \leq0.
\end{align*}
Thus, $h(t)\leq h(s)$ for all $t\in\lbrack s,I(z,s))$, proving
\eqref{claim}.

Now,  we prove that
\begin{equation}\label{step1}
I(z,s)=I(w,s)\quad \hbox{for all $z,w\in\mathbb{D}$ and
$s\geq0.$}
\end{equation}
Fix $s\geq0.$ Suppose that there are two points $z,w$ such that
$I(z,s)<I(w,s).$ Thus, letting $t\to I(z,s)$, we have
$\phi_{s,z}(t)\to \de \D$, while $\phi_{s,w}(t)$ stays compact
inside $\D$. In particular,
$\rho_{\mathbb{D}}(\phi_{s,z}(t),\phi_{s,w}(t))\to \infty$,
which contradicts  \eqref{claim}. Thus $I(z,s)\geq I(w,s).$
Swapping the role of $z,w$ in the previous argument, we have
\eqref{step1}.

Next, let $I=I(0,0)$.  We prove that
\begin{equation}\label{step2}
    I(0,s)=I \quad \hbox{for all $s<I.$}
\end{equation}
Let $s<I$ and $z=\phi_{0,0}(s)\in \D$. Take
$s<t<\min\{I,I(0,s)\}.$ By \eqref{step1}, $I(z,s)=I(0,s)$ so
$\phi_{s,z}(t)$ is well-defined and belongs to $\D$. Therefore,
by uniqueness of solutions of ODE's it follows that
\[
\phi_{s,z}(t)=\phi_{s,\phi_{0,0}(s)}(t)=\phi_{0,0}(t).
\]
Moreover, by \eqref{claim},
\begin{equation*}
\rho_{\mathbb{D}}(\phi_{s,0}(t),\phi_{0,0}(t))  =\rho_{\mathbb{D}}%
(\phi_{s,0}(t),\phi_{s,z}(t))
\leq\rho_{\mathbb{D}}(0,z)=\rho_{\mathbb{D}}(0,\phi_{0,0}(s)).
\end{equation*}
From this, arguing as in the proof of \eqref{step1}, equation
\eqref{step2} follows.

Finally we prove that there exists $\delta>0$ such that
\begin{equation}\label{step3}
    I(0,s)\geq s+\delta \quad\hbox{ for all $s\in\lbrack0,I).$}
\end{equation}
Fix $0<r<1.$ We know that there exists a non-negative function
$k_{r,I+2}\in L^{d}([0,I+2],\mathbb{R})$ such that
\[
|G(z,t)|\leq k_{r,I+2}(t)
\]
for all $|z|\leq r$ and for almost every $t\in\lbrack0,I+2].$
Moreover, by Lemma \ref{WHVF4}, there exists a non-negative
function $\widehat{k}_{r,I+2}\in L^{d}([0,T],\mathbb{R})$ such
that
\[
|G(z,t)-G(w,t)|\leq\widehat{k}_{r,I+2}(t)|z-w|
\]
for all $|z|,|w|\leq r$ and for almost every
$t\in\lbrack0,I+2].$ The functions $\lbrack0,I+2]\ni
u\mapsto\int_{0}^{u}k_{r,I+2}(t)dt$ and $\lbrack 0,I+2]\ni u
\mapsto\int_{0}^{u}\widehat{k}_{r,I+2}(t)dt$ are absolutely
continuous and therefore there exists $0<\delta<1$ such that
for all $s\in\lbrack0,I+1]$ it holds
$\int_{s}^{s+\delta}k_{r,I+2}(t)dt\leq r$ and
$\int_{s}^{s+\delta }\widehat{k}_{r,I+2}(t)dt\leq r.$ Moreover,
if $f:[s,s+\delta]\rightarrow r\overline{\mathbb{D}}$ is
measurable, then $\lbrack s,s+\delta]\ni\xi\mapsto$
$G(f(\xi),\xi)$ is integrable and
\[
\left\vert \int_{s}^{t}G(f(\xi),\xi)d\xi\right\vert \leq\int_{s}%
^{t}|G(f(\xi),\xi)|d\tau\leq\int_{s}^{t}k_{r,I+2}(\xi)d\xi\leq
r.
\]
Therefore, for $s\in\lbrack0,I+1],$ we can define by induction
\begin{equation}
\left\{
\begin{array}
[c]{l}%
x_{s,0}(t):=0\\
x_{s,n}(t):=\int_{s}^{t}G(x_{s,n-1}(\xi),\xi)d\xi
\end{array}
\right. \label{induction}%
\end{equation}
for $t\in\lbrack s,s+\delta]$ and $n\in\mathbb{N}.$ Now, since $|x_{s,n}%
(t)|\leq r,$ we have
\begin{align*}
|x_{s,n}(t)-x_{s,n-1}(t)|  & \leq\int_{s}^{t}|G(x_{s,n-1}(\xi
),\xi)-G(x_{s,n-2}(\xi),\xi)|d\xi\\
&
\leq\int_{s}^{t}\widehat{k}_{r,I+2}(\xi)|x_{s,n-1}(\xi)-x_{s,n-2}(\xi
)|d\xi\\
& \leq\max_{\xi\in\lbrack
s,s+\delta]}|x_{s,n-1}(\xi)-x_{s,n-2}(\xi)|\int
_{s}^{t}\widehat{k}_{r,I+2}(\xi)d\xi\\
& \leq r\max_{\xi\in\lbrack
s,s+\delta]}|x_{s,n-1}(\xi)-x_{s,n-2}(\xi)|.
\end{align*}
From this inequality, we deduce that $\{x_{s,n}\}$ is a Cauchy
sequence in the Banach space $C([s,s+\delta])$ of continuous
complex functions from $[s,s+\delta]$, endowed with the
supremum norm. Therefore, it converges uniformly on
$[s,s+\delta]$ to a function $x\in C([s,s+\delta]).$ Since
\[
|G(x_{s,n-1}(\tau),\tau)|\leq k_{r,I+2}(\tau),
\]
 the Lebesgue dominated converge
theorem implies
\[
x(t)=\int_{s}^{t}G(x(\xi),\xi)d\xi
\]
for all $t\in\lbrack s,s+\delta].$ Therefore, $\phi_{s,0}=x$ on
$[s,s+\delta ],$ which proves that $I(0,s)\geq s+\delta$,
proving \eqref{step3}.

Equation \eqref{forever} follows immediately from
\eqref{step1}, \eqref{step2} and \eqref{step3}, and we are
done.
\end{proof}

As we will see, Herglotz vector fields in the unit disc can be
decomposed by means of Herglotz functions (this the the reason
for the name). We begin by recalling the following definition:

\begin{definition}\label{def-Her-fun}
Let $d\in [1,+\infty]$. A {\sl Herglotz function of order $d$}
is a function $p:\mathbb{D}\times\lbrack0,+\infty
)\mapsto\mathbb{C}$ with the following properties:

\begin{enumerate}
\item[HF1.] For all $z\in\mathbb{D},$ the function $\lbrack0,+\infty
)\ni t \mapsto p(z,t)\in\mathbb{C}$ belongs to
$L_{loc}^{d}([0,+\infty),\mathbb{C})$;

\item[HF2.] For all $t\in\lbrack0,+\infty),$ the function
$\mathbb{D}\ni z \mapsto p(z,t)\in\mathbb{C}$ is holomorphic;

\item[HF3.] For all $z\in\mathbb{D}$ and for all $t\in\lbrack0,+\infty),$ we
have $\Re p(z,t)\geq0.$
\end{enumerate}
\end{definition}

\begin{proposition}
\label{integrabilidad=integrabilidaden1punto} Let
$d\in\lbrack1,\infty].$ A function
$p:\mathbb{D}\times\lbrack0,+\infty)\mapsto\mathbb{C}$ is a
Herglotz function of order $d$ if and only if it satisfies HF2,
HF3 and the following two statements:

\begin{enumerate}
\item for all $z\in\mathbb{D},$ the function $\lbrack0,+\infty)\ni t\mapsto
p(z,t)\in\mathbb{C}$ is measurable;

\item there exists $z_{0}\in\mathbb{D}$ such that the function $\lbrack0,+\infty)\ni t \mapsto p(z_{0},t)\in\mathbb{C}$ belongs to
$L_{loc}^{d
}([0,+\infty),\mathbb{C})$.
\end{enumerate}
\end{proposition}

\begin{proof}
We  have to prove that if $p$ satisfies HF2, HF3 and (1) and
(2), then it satisfies HF1. Let $z\in\mathbb{D}$. Fix a point
$t\geq0.$ Bearing in mind that the map $\mathbb{D}\ni w\mapsto
p(w,t)\in\mathbb{C}$ is holomorphic, by \cite[pages
39-40]{Pommerenke}, we have that
\[
|p(z,t)|\leq\frac{1+|z|}{1-|z|}|p(0,t)|\leq\frac{1+|z|}{1-|z|}\frac{1+|z_{0}%
|}{1-|z_{0}|}|p(z_{0},t)|.
\]
Now, since the function $\lbrack0,+\infty)\ni t\mapsto
p(z_{0},t)$ belongs to $L_{loc}^{d}([0,+\infty),\mathbb{C})$
and $\lbrack0,+\infty)\ni t\mapsto p(z,t)$ is measurable, the
above inequality implies that the function $\lbrack0,+\infty)\ni t\mapsto p(z,t)$ also belongs to $L_{loc}^{d}%
([0,+\infty),\mathbb{C}).$
\end{proof}

We are going to show  that there is essentially a one-to-one
correspondence between Herglotz vector fields and Berkson-Porta
data. To this aim we need a lemma:

\begin{lemma}\label{laszlo}
Let $G:\D\times [0,+\infty)  \to \C$ be a function such that
\begin{enumerate}
  \item For all $t\geq 0$ the map $\D\ni z\mapsto G(z,t)$
  is holomorphic.
  \item For all $z\in \D$ the map $[0,+\infty) \ni t\mapsto
  G(z,t)$ is measurable.
\end{enumerate}
Then the map $[0,+\infty)\ni t\mapsto G(\cdot, t)\in {\sf
Hol}(\D,\C)$, from the set $[0,+\infty)$ endowed with the
Lebesgue measure to the Fr\'echet space ${\sf Hol}(\D,\C)$, is
measurable.
\end{lemma}

\begin{proof}
Since  ${\sf Hol}(\D,\C)$ is a metrizable and separable
topological space, it is enough to show that, given $f\in {\sf
Hol}(\D,\C)$ and $\epsilon>0$, the set
\[
\{t\in [0,+\infty) : \hbox{d}_H(G(\cdot, t), f)<\epsilon\}
\]
is measurable; where here $\hbox{d}_H(\cdot, \cdot)$ denotes
the Fr\'echet distance in ${\sf Hol}(\D,\C)$.

Fix $t\geq 0$. Since $G(\cdot, t)$ is holomorphic in $\D$ there
exists a sequence $\{g_n(t)\}\subset \C$ such that
\[
G(z,t)=\sum_{n=0}^\infty g_n(t) z^n.
\]
The functions $[0,+\infty)\ni t\mapsto g_n(t)$ are measurable.
Indeed, $g_0(t)=G(0,t)$ is measurable by hypothesis (2). By
induction, assume that $g_k(t)$ is measurable for $k=0,\ldots,
n$. Then
\begin{equation*}
\begin{split}
g_{n+1}(t)&=\frac{G^{(n+1)}(0,t)}{(n+1)!}=\lim_{h\to
0}\frac{1}{h^{n+1}}[G(h,t)-\sum_{k=0}^n
\frac{G^{(k)}(0,t)}{k!}h^k]\\&= \lim_{m\to \infty} m^{n+1}[
G(\frac{1}{m},t)-\sum_{k=0}^n g_k(t)(\frac{1}{m})^k],
\end{split}
\end{equation*}
which proves that $t\mapsto g_{n+1}(t)$ is measurable,
concluding the induction.

Let $G_m(z,t):=\sum_{n=0}^m g_n(t) z^n$. Since the map $\C^m\ni
(a_0, \ldots, a_m)\mapsto \sum_{n=0}^m a_n z^n\in {\sf
Hol}(\D,\C)$ is continuous, and $t\mapsto g_n(t)$ is measurable
for all $n\in \N$, then $[0,+\infty)\ni t\mapsto G_m(\cdot
,t)\in {\sf Hol}(\D,\C)$ is measurable for all $m\in \N$.
Moreover, $\{G_m(\cdot, t)\}$ converges to $G(\cdot, t)$ in
${\sf Hol}(\D,\C)$. Therefore, $\hbox{d}_H(G(\cdot, t),
f)=\lim_{m\to \infty} \hbox{d}_H(G_m(\cdot, t), f)$, and hence
it follows easily that
\begin{equation*}
\begin{split}
\{t\in [0,+\infty) &: \hbox{d}_H(G(\cdot, t), f)<\epsilon\}
\\=& \bigcup_{p=1}^\infty
\bigcup_{n=1}^\infty\bigcap_{m=n}^\infty \{t\in [0,+\infty) :
\hbox{d}_H(G_m(\cdot, t), f)\leq \epsilon(1-\frac{1}{p+1})\}.
\end{split}
\end{equation*}
Since $\{G_m(\cdot, t)\}$ are measurable,  $\{t\in [0,+\infty)
: \hbox{d}_H(G_m(\cdot, t), f)\leq \epsilon(1-\frac{1}{p+1})\}$
is measurable and this proves the result.
\end{proof}

\begin{theorem}
\label{Herglotz-implica-VF} Let
$\tau:[0,+\infty)\rightarrow\overline {\mathbb{D}}$ be a
measurable function and let $p:\D\times [0,+\infty)\to \C$ be a
Herglotz function of order $d\in [1,+\infty)$. Then the map
$G_{\tau,p}:\mathbb{D}\times\lbrack
0,+\infty)\rightarrow\mathbb{C}$ given by
\begin{equation}\label{Herglotz-vf}
G_{\tau,p}(z,t)=(z-\tau(t))(\overline{\tau(t)}z-1)p(z,t),
\end{equation}
for all $z\in\mathbb{D}$ and for all $t\in\lbrack0,+\infty),$
is a Herglotz vector field  of order $d$ on the unit disc.

Conversely, if $G:\D\times [0,+\infty)\to \C$ is a Herglotz
vector field of order $d\in [1,+\infty)$ on the unit disc, then
there exist a measurable function
$\tau:[0,+\infty)\rightarrow\overline {\mathbb{D}}$ and a
Herglotz function $p:\D\times [0,+\infty)\to \C$  of order $d$
such that $G(z,t)=G_{\tau, p}(z,t)$ for almost every $t\in
[0,+\infty)$ and all $z\in \D$ (here $G_{\tau, p}$ is given by
\eqref{Herglotz-vf}).

Moreover, if $\tilde{\tau}:[0,+\infty)\rightarrow\overline
{\mathbb{D}}$ is another measurable function and
$\tilde{p}:\D\times [0,+\infty)\to \C$ is another Herglotz
function of order $d$ such that $G=G_{\tilde{\tau}, \tilde{p}}$
for almost every $t\in [0,+\infty)$ then
$p(z,t)=\tilde{p}(z,t)$ for almost every $t\in [0,+\infty)$ and
all $z\in \D$ and $\tau(t)=\tilde{\tau}(t)$ for almost all
$t\in [0,+\infty)$ such that $G(\cdot, t)\not\equiv 0$.
\end{theorem}

\begin{proof} Assume $(\tau, p)$ be given. By the Berkson-Porta representation formula, for each fixed $t\in [0,+\infty)$ the function
$\D \ni z\mapsto G_{\tau,p}(z,t)$ is an infinitesimal
generator. Thus we need to prove that $G_{\tau, p}$ is a weak
holomorphic vector field of order $d$ on $\D$.

On the one hand, it is clear that for all $z\in\mathbb{D},$ the
function $\lbrack0,+\infty)\ni t \mapsto G_{\tau,p}(z,t)$ is
measurable and that for all $t\in\lbrack0,+\infty),$ the
function $\mathbb{D}\ni z\mapsto G_{\tau,p}(z,t)$ is
holomorphic. That is, $G_{\tau,p}$ satisfies WHVF1 and WHVF2.
On the other hand, fix a compact set $K\subset\mathbb{D}$ and
$T>0$. Let $0<r<1$ be such that $K\subset \D(r)=\{\zeta\in \D:
|\zeta|<r\}$. Fix $z\in K$ and $t\in \lbrack0,T].$ By
\cite[pages 39-40]{Pommerenke},
\begin{align*}
|G_{\tau,p}(z,t)|  & =|(z-\tau(t))(\overline{\tau(t)}z-1)||p(z,t)|\leq
4|p(z,t)|\\
& \leq4\frac{1+|z|}{1-|z|}|p(0,t)|\leq4\frac{1+r}{1-r}|p(0,t)|.
\end{align*}
Since the function $\lbrack0,+\infty)\ni t\mapsto p(0,t)$
belongs to $L_{loc}^{d}([0,+\infty),\mathbb{C}),$ writing
$k_{K,T}(t)=4\frac {1+r}{1-r}|p(0,t)|$ we conclude that
$G_{\tau,p}$ satisfies WHVF3 and it is a weak holomorphic
vector field of order $d$.

Conversely, let $G$ be a Herglotz vector field. Hence $z\mapsto
G(z,s)$ belongs to $\mathrm{Gen}(\mathbb{D}),$ for almost every
$s\in\lbrack0,+\infty)$. Therefore, by the Berkson-Porta
representation formula, we can find $\alpha_{s}\in
\overline{\mathbb{D}}$ and
$p_{s}\in\mathrm{Hol}(\mathbb{D},\mathbb{C})$ with $\Re
p_{s}\geq0$ such that, for all $z\in\mathbb{D}$ and almost
every $s\in [0,+\infty)$
\[
G(z,s)=(z-\alpha_{s})(\overline{\alpha_{s}}z-1)p_{s}(z).
\]
By WHVF1 for each fixed $z\in \D$ the function $[0,+\infty)\ni
t \mapsto G(z,t)$ is measurable. By Lemma \ref{laszlo}, the map
$\Psi:[0,+\infty)\ni t\mapsto G(\cdot, t)\in {\rm Hol}(\D,\C)$
from the set $[0,+\infty)$ endowed with the Lebesgue measure to
the Fr\'echet space ${\rm Hol}(\D,\C)$, is measurable.

Note that if $G(\cdot,s)\equiv 0$ then necessarily
$p_{s}(\cdot)\equiv 0$ and in such a case $\alpha_s$ can take
any value. We set $\alpha_s= 0$ if $G(\cdot,s)\equiv 0$.
 Let
$E:=\{s\in\lbrack0,+\infty):G(\cdot,s)=\underline{0}\}$ that
is, $s\in E$ if and only if $G(\cdot,s)\equiv 0$. Note that
since $E=\Psi^{-1}(\{\underline{0}\})$, the set $E$ is a
measurable subset of $[0,+\infty)$. Hence, $\alpha_s=0$ for
$s\in E$ and $\alpha_s=BP_{\tau}\circ\Psi(s)$ for $s\in
[0,+\infty)\setminus E$. Since $E$ is measurable, $\Psi$ is
measurable, $BP_{\tau}$ is continuous by Proposition
\ref{BP-continuidad} and $\mathrm{Gen}(\mathbb{D})$ is a closed
subset of $\mathrm{Hol}(\mathbb{D},\mathbb{C})$, it follows
that $\alpha_s$ is a measurable mapping from $[0,+\infty)$
into~$\overline {\mathbb{D}}.$

Similarly, being $p_s(z)\equiv 0$ for $s\in E$ and
$p_s(z)=BP_{p}\circ\Psi(s)$ for $s\in [0,+\infty)\setminus E$
and being $BP_{p}$  continuous by Proposition
\ref{BP-continuidad} we deduce that $p_{s}$ is a measurable map
from $[0,+\infty)$ into $\mathrm{Hol}(\mathbb{D},\mathbb{C})$.

We are left to check that $p_s$ is a Herglotz function of order
$d$. By Proposition \ref{integrabilidad=integrabilidaden1punto}
this is equivalent to show that there exists a point
$z_{0}\in\mathbb{D}$ such that the mapping
$\lbrack0,+\infty)\ni s\mapsto p_s(0)\in\mathbb{C}$ belongs to
$L_{loc}^{d}([0,+\infty),\mathbb{C}).$

Let $A:=\{s\in\lbrack0,+\infty):|\alpha_s|\geq\frac{1}{2}\}$.
Since $A=\alpha_{s}^{-1}(\mathbb{D}\setminus \D(1/2))$, we see
that $A$ is a Lebesgue measurable subset of $[0,+\infty).$
Moreover, when $s\in A$ clearly $\alpha_s\neq0$ and
\[
|p_s(0)|=\frac{\left\vert G(0,s)\right\vert }{|\alpha_s|}
\leq2\left\vert G(0,s)\right\vert.
\]
Hence $ \lbrack0,+\infty)\ni s\mapsto
\chi_{A}(s)p_s(0)\in\mathbb{C}$ belongs to
$L_{loc}^{d}([0,+\infty),\mathbb{C})$, where $\chi_{A}(s)=1$
for $s\in A$ and $\chi_{A}(s)=0$ otherwise.

Moreover, by the very definition of $A,$ when $s\in
[0,+\infty)\setminus A,$
\begin{align*}
|G(3/4,s)|  & =\left\vert 3/4-\alpha_s\right\vert \left\vert
\overline
{\alpha_s}3/4-1\right\vert \left\vert p_s(3/4)\right\vert \\
& \geq\left(  \frac{3}{4}-\frac{1}{2}\right)  \left(
1-\frac{3}{4}\right)
\left\vert p_s(3/4)\right\vert \\
& =\frac{1}{16}\left\vert p_s(3/4)\right\vert .
\end{align*}
Hence $\lbrack0,+\infty)\ni s\mapsto \chi_{[0,+\infty)\setminus
A}(s)p_s(3/4)\in\mathbb{C}$ belongs to
$L_{loc}^{d}([0,+\infty),\mathbb{C}).$

By the distortion theorem for Carath\'{e}odory functions
\cite[pages 39-40]{Pommerenke},  for every $s\in\lbrack
0,+\infty)$,
\[
|p_s(0)|\leq\frac{1+3/4}{1-3/4}|p_s(3/4)|=7|p_s(3/4)|.
\]
Therefore,   $\lbrack0,+\infty)\ni s\mapsto
\chi_{[0,+\infty)\setminus A} (s)p_s(0)\in\mathbb{C}$ belongs
to $L_{loc}^{d}([0,+\infty),\mathbb{C})$. Thus $p_s(0)\in
L_{loc}^{d}([0,+\infty),\mathbb{C})$, and we are done.

The statement about uniqueness follows at once from the
uniqueness of the Berkson-Porta representation formula.
\end{proof}

The representation of Herglotz vector fields by means of
Herglotz functions given by Theorem \ref{Herglotz-implica-VF}
will turn out to be a very powerful tool, because it allows to
use distortion theorems for Carath\'eodory's function, a tool
which is not available in higher dimensions (see \cite{BCM2}).

\section{From Herglotz vector fields to evolution families}

For the sake of \ clearness, we begin by recalling the
well-known Gronwall's Lemma as needed for our aims.

\begin{lemma}
\label{Gronwall} Let $\theta:[a,b]\rightarrow\mathbb{R}$ be a
continuous function and $k\in L^{1}([a,b],\mathbb{R)}$
non-negative. If there exists $C\geq0$ such that for all
$t\in\lbrack a,b]$
\[
\theta(t)\leq C+\int_{a}^{t}\theta(\xi)k(\xi)d\xi\qquad\text{ (resp., }%
\theta(t)\leq C+\int_{t}^{b}\theta(\xi)k(\xi)d\xi),
\]
then%
\[
\theta(t)\leq C\exp\left(  \int_{a}^{t}k(\xi)d\xi\right)  \qquad\text{ (resp.,
}\theta(t)\leq C\exp\left(  \int_{t}^{b}k(\xi)d\xi\right)  \text{)}.
\]

\end{lemma}

\begin{theorem}
\label{Herglotz-implica-EF} Let $G:\D\times[0,+\infty)\to \C$
be a Herglotz vector field of order $d\in [1,+\infty)$.  For
all $s\geq0$ and $z\in\mathbb{D},$ let $\phi_{s,z}$ be the
solution of the problem
\[
\left\{
\begin{array}
[c]{l}%
\overset{\bullet}{x}(t)=G(x(t),t)\text{ for a. e. }t\in\lbrack
s,+\infty)\\
x(s)=z.
\end{array}
\right.  \text{ }%
\]
Let $\varphi_{s,t}(z):=\phi_{s,z}(t)$ for all $0\leq s\leq
t<+\infty$ and for all $z\in\mathbb{D}$. Then $(\varphi_{s,t})$
is an evolution family in the unit disc of order $d$.
\end{theorem}

\begin{proof}
By Theorem \ref{semicompletezza}, the Herglotz vector field $G$
is a semi-complete weak holomorphic vector field on the unit
disc. Therefore, the value $\phi _{s,z}(t)$ is well-defined for
all $0\leq s\leq t<+\infty$ and for all $z\in\mathbb{D}.$
Moreover, by uniqueness of solutions of ODE's, it follows that
$\varphi_{s,t}=\varphi_{u,t}\circ\varphi_{s,u}$ for all $0\leq
s\leq u\leq t<+\infty.$ Thus EF1, EF2 hold, and we are left to
prove EF3 and the holomorphicity of $\varphi_{s,t}.$

We prove that $\varphi_{s,t}:\mathbb{D}\rightarrow\mathbb{D}$
is holomorphic for all $0\leq s\leq u\leq t<+\infty$.

First, we claim that for each $0<T<+\infty$ and $0<r<1,$ there
exists $R=R(r,T)<1$ such that
\begin{equation}\label{claim2}
|\varphi_{s,t}(z)|\leq R
\end{equation}
for all $0\leq s\leq t\leq T$ and $|z|\leq r.$

Seeking for a contradiction, assume \eqref{claim2} is not true.
Then there exist three sequences $(z_{n}),$ $(s_{n}),$ and
$(t_{n})$ such that $|z_{n}|\leq r,$ $z_{n}\rightarrow z_{0},$
$s_{n},t_{n}\in\lbrack0,T],$ $s_{n}\leq t_{n},$
$s_{n}\rightarrow s_{0},$ $t_{n}\rightarrow t_{0},$ and
$|\varphi_{s_{n},t_{n}}(z_{n})|\rightarrow1.$ Since the map
$\varphi _{s_{n},t_{n}}$ is a contraction for the hyperbolic
metric (see the proof of
Theorem \ref{Herglotz-implica-VF}), we have that $\rho_{\mathbb{D}}%
(\varphi_{s_{n},t_{n}}(z_{n}),\varphi_{s_{n},t_{n}}(z_{0}))\leq\rho
_{\mathbb{D}}(z_{n},z_{0})\rightarrow0.$ Then $|\varphi_{s_{n},t_{n}}%
(z_{0})|\rightarrow1.$ The map $t\mapsto\varphi_{0,t}(z_{0}) $
is continuous (because $\phi_{0,z_{0}}$ is a positive
trajectory of the semi-complete vector field $G$). Moreover,
\begin{align*}
\rho_{\mathbb{D}}(\varphi_{0,t_{n}}(z_{0}),\varphi_{s_{n},t_{n}}(z_{0}))
&
=\rho_{\mathbb{D}}(\varphi_{s_{n},t_{n}}(\varphi_{0,s_{n}}(z_{0}%
)),\varphi_{s_{n},t_{n}}(z_{0}))\\
&
\leq\rho_{\mathbb{D}}(\varphi_{0,s_{n}}(z_{0}),z_{0})\rightarrow
\rho_{\mathbb{D}}(\varphi_{0,s}(z_{0}),z_{0})<+\infty.
\end{align*}
Again this implies that
$|\varphi_{0,t_{n}}(z_{0})|\rightarrow1.$ But
$\varphi_{0,t_{n}}(z_{0})\rightarrow\varphi_{0,t}(z_{0})\in\mathbb{D}$.
A contradiction. Hence \eqref{claim2} holds.

Fix $s<t$ and $z\in\mathbb{D}$. Let $|z|<r<1$, $T>t$ and let
$R=R(r,T)$ be given in \eqref{claim2}. Write
$\widehat{R}=(R+1)/2.$ By the very definition of weak
holomorphic vector field and by Lemma \ref{WHVF4} there
exist two non-negative functions $k_{R,T},\widehat{k}_{R,T}\in L^{d}%
([0,T],\mathbb{R})$ such that
\[
|G(w,u)|\leq k_{R,T}(u)
\]
and
\[
|G(w_{1},u)-G(w_{2},u)|\leq\widehat{k}_{R,T}(u)|w_{1}-w_{2}|
\]
for all $|w_{1}|,|w_{2}|,|w|\leq R$ and for almost every
$u\in\lbrack0,T].$

The map $u\mapsto G^{\prime}(z,u)$ is clearly measurable. Thus,
the function $t\mapsto G^{\prime}(\varphi_{s,t}(z),t)$ is also
measurable. Therefore,
\[
|G^{\prime}(\varphi_{s,u}(z),u)|=\frac{1}{2\pi}\left\vert
{\displaystyle\int\nolimits_{C(0,\widehat{R})^{+}}}
\frac{G(\xi,u)}{\xi-\varphi_{s,u}(z)}d\xi\right\vert \leq k_{R,T}(u)\frac
{2R}{1-R}.
\]
Therefore, the map $u\mapsto G^{\prime}(\varphi_{s,u}(z),u)$
belongs to $L^{d}([s,T],\mathbb{R}).$ Once we know that this
function is integrable, we claim that
\[
\lim_{h\rightarrow0}\frac{\varphi_{s,u}(z+h)-\varphi_{s,u}(z)}{h}=\exp\left(
\int_{s}^{t}G^{\prime}(\varphi_{s,u}(z),u)du\right)  .
\]
To simplify the notation we write $H(u):=\exp\left(
-\int_{s}^{u}G_{\tau
,p}^{\prime}(\varphi_{s,\xi}(z),\xi)d\xi\right)  .$ Moreover,
for $|h|<R-|z|$ we define
$\theta(u):=|\varphi_{s,u}(z+h)-\varphi_{s,u}(z)|$ and
$f_{h}(u):=\frac{\varphi_{s,u}(z+h)-\varphi_{s,u}(z)}{h}.$ We
have
\begin{align*}
\theta(u)  & =\left\vert h+\int_{s}^{u}G(\varphi_{s,\xi}%
(z+h),\xi)d\xi-\int_{s}^{u}G(\varphi_{s,\xi}(z),\xi)d\xi\right\vert
\\
& \leq|h|+\int_{s}^{u}\theta(\xi)\widehat{k}_{R,T}(\xi)d\xi=\theta(s)+\int
_{s}^{u}\theta(\xi)\widehat{k}_{R,T}(\xi)d\xi.
\end{align*}
Lemma \ref{Gronwall} implies that
\[
\theta(u)\leq\theta(s)\exp\left(  \int_{s}^{u}\widehat{k}_{R,T}(\xi
)d\xi\right)  =|h|\exp\left(  \int_{s}^{u}\widehat{k}_{R,T}(\xi)d\xi\right)  .
\]
That is, $|f_{h}(u)|\leq\exp\left(  \int_{s}^{u}\widehat{k}_{R,T}(\xi
)d\xi\right)  .$ In a similar way, for all $0\leq s\leq v\leq u\leq t,$ we
have that
\begin{align*}
\theta(v)  & \leq\theta(u)+\left\vert
\int_{v}^{u}G(\varphi_{s,\xi
}(z+h),\xi)d\xi-\int_{v}^{u}G(\varphi_{s,\xi}(z),\xi)d\xi\right\vert
\\
& \leq\theta(u)+\int_{v}^{u}\theta(\xi)\widehat{k}_{R,T}(\xi)d\xi.
\end{align*}
Using again Lemma \ref{Gronwall}, we have that
\[
|\varphi_{s,v}(z+h)-\varphi_{s,v}(z)|\leq|\varphi_{s,u}(z+h)-\varphi
_{s,u}(z)|\exp\left(  \int_{v}^{u}\widehat{k}_{R,T}(\xi)d\xi\right)  \text{
for all }s\leq v\leq u\leq t.
\]
In particular,
\begin{equation}
|h|\leq|\varphi_{s,u}(z+h)-\varphi_{s,u}(z)|\exp\left(  \int_{s}^{t}%
\widehat{k}_{R,T}(\xi)d\xi\right)  .\label{Eq_Univalencia}%
\end{equation}
This means that if $h\neq0,$ then
$\varphi_{s,u}(z+h)\neq\varphi_{s,u}(z) $ for all $u\in\lbrack
s,t].$ Fix $h>0.$ Then there is a  set $A=A(h)$ of zero measure
such that for all $u\in\lbrack s,T]\setminus A(h),$ we have
\begin{align*}
f_{h}^{\prime}(u)  & =\frac{\overset{\bullet}{\varphi}_{s,u}(z+h)-\overset
{\bullet}{\varphi}_{s,u}(z)}{h}\\
& =\frac{G(\varphi_{s,u}(z+h),u)-G(\varphi_{s,u}(z),u)}{h}\\
& =G^{\prime}(\varphi_{s,u}(z),u)f_{h}(u)+L_{h}(u)
\end{align*}
where
\[
L_{h}(u)=f_{h}(u)\left[  \frac{G(\varphi_{s,u}(z+h),u)-G_{\tau
,p}(\varphi_{s,u}(z),u)}{\varphi_{s,u}(z+h)-\varphi_{s,u}(z)}-G_{\tau
,p}^{\prime}(\varphi_{s,u}(z),u)\right]  .
\]
Note $L_{h}(u)$ is well-defined because $\varphi_{s,u}(z+h)\neq\varphi
_{s,u}(z).$

Then, by the very definition of $H,$ it holds
\[
\frac{d(H(u)f_{h}(u))}{du}=H(u)f_{h}^{\prime}(u)-H(u)G^{\prime
}(\varphi_{s,u}(z),u)f_{h}(u)=H(u)L_{h}(u)\text{ \quad a.e. on
}u\in\lbrack s,t].
\]
Integrating on $u\in\lbrack s,t],$ we obtain
\[
H(t)f_{h}(t)-1=\int_{s}^{t}H(u)L_{h}(u)du.
\]
Moreover
\begin{multline*}
|H(u)L_{h}(u)|\leq\\
\leq\left\vert \exp\left(
-\int_{s}^{u}G^{\prime}(\varphi_{s,\xi
}(z),\xi)d\xi\right)  \right\vert \exp\left(  \int_{s}^{u}\widehat{k}%
_{R,T}(\xi)d\xi\right)  \left[  \widehat{k}_{R,T}(u)+k_{R,T}(u)\frac{2R}%
{r-1}\right]  .
\end{multline*}
Since this bound does not depend on $h$ and
\[
\lim_{h\rightarrow0}\frac{G(\varphi_{s,u}(z+h),u)-G%
(\varphi_{s,u}(z),u)}{\varphi_{s,u}(z+h)-\varphi_{s,u}(z)}=G^{\prime
}(\varphi_{s,u}(z),u),
\]
by the Lebesgue dominated convergence theorem, we have $\lim_{h\rightarrow
0}\int_{s}^{t}H(u)L_{h}(u)du=0.$ Therefore,
\[
\lim_{h\rightarrow0}\frac{\varphi_{s,u}(z+h)-\varphi_{s,u}(z)}{h}=\frac
{1}{H(t)},
\]
proving that $\varphi_{s,u}(z)$ is holomorphic for all $0\leq
s\leq t<+\infty$.

To end up the proof we need to check property EF3. Let $0\leq
s\leq u\leq t\leq T$, $z\in\mathbb{D}$ and let $R=R(T,|z|)$ be
the number given by \eqref{claim2}. Then
\begin{align*}
|\varphi_{s,u}(z)-\varphi_{s,v}(z)|  & =\left\vert \int_{u}^{v}\overset
{\bullet}{\varphi}_{s,\xi}(z)d\xi\right\vert \\
& =\left\vert
\int_{u}^{v}G(\varphi_{s,\xi}(z),\xi)d\xi\right\vert
\leq\int_{u}^{v}k_{R,T}(\xi)d\xi.
\end{align*}
Note that this also implies that if $p$ is of order $d$ for
some $d\in\lbrack1,+\infty],$ then $(\varphi_{s,t})$ is also of
order $d$.
\end{proof}

\section{From evolution families to Herglotz vector fields}

In this section we prove the converse of Theorem
\ref{Herglotz-implica-EF}. Part of the proof relies on the
following result on measurable selections:

\begin{theorem}
\label{seleccion medible} \cite[Theorem III.30, page 80]{Castaing-Valadier}
Let $(\Omega,\Sigma,\mu)\,\,$be a positive $\sigma$-finite complete measure
space, $[X,d]$\ a separable and complete metric space and $\Gamma$\ a
multifunction from $\Omega$\ to the subsets of $X.$\ Assume that:

(i) For every $\omega\in\Omega,$\ $\Gamma(\omega)$\ is a closed
non-empty subset of $X$.

(ii) For every $x\in X$\ and every $r>0,$\
$\{\omega\in\Omega:\Gamma (\omega)\cap
B(x,r)\neq\emptyset\}\in\Sigma.$\ (As usual, $B(x,r)$\ denotes
the open unit ball in $X$\ with center $x$\ and radius $r).$

Then $\Gamma$\ admits a measurable selector
$\sigma:\Omega\longrightarrow X$; namely, for every
$\omega\in\Omega,$\ we have $\sigma(\omega)\in \Gamma(\omega)$\
and the inverse image by $\sigma$\ of any borelian in $X$\
belongs to $\Sigma.$
\end{theorem}

Now we are going to prove the main result of this section:

\begin{theorem}
\label{EF-implica-VF} Let $(\varphi_{s,t})$ be an evolution
family of order $d$ in the unit disc. Then there exists a
Herglotz vector field $G$ which has positive trajectories
$(\varphi_{s,t})$; namely, for any $(z,s)\in\mathbb{D}\times
\lbrack0,+\infty)$, the positive trajectory of the vector field
$G$ with initial data $(z,s)$ is exactly $\lbrack s,+\infty)\ni
t\rightarrow\varphi_{s,t}(z)$.
\end{theorem}

\begin{proof}
The proof of this theorem is rather long and has three main
parts which will be exposed separately. In short: (a)
construction of a candidate function
$G:\mathbb{D}\times\lbrack0,+\infty)$ $\mathbb{\rightarrow C}$
verifying that $G(\cdot,s)\in \mathrm{Gen}(\mathbb{D}),$ for
all $s\geq0$; (b) checking that $G$ is a weak holomorphic
vector field; (c) verification of the assertion of the theorem.

\textit{Part (a):} We are going to apply Theorem \ref{seleccion
medible} to a suitably chosen $\Gamma: [0,+\infty)\to
2^{\rm{Hol}(\D,\C)}$, where the set $[0,+\infty)$ is endowed
with the Lebesgue measure and ${\rm{Hol}(\D,\C)}$ has its
natural structure of Fr\'echet space.

Fix $z\in\mathbb{D}$ and $T>0$. Let $k:=k_{z,T}\in
L^{d}([0,T+1],\mathbb{R})$ be the non-negative function given
by EF3. We extend $k$ to all of $\R$ by setting zero outside
the interval $[0,T+1]$. Then for $0\leq s\leq T$ and every
$n\in\mathbb{N}$
\begin{equation*}
n\left\vert
\varphi(z,s,s+\frac{1}{n})-\varphi(z,s,s)\right\vert \leq
n\int_{s}^{s+1/n}k(\xi)d\xi\leq {\sf Max}_{k}(s),
\end{equation*}
where
\[
{\sf Max}_{k}(s):=\sup\left\{
\frac{1}{|I|}\int_{I}k(\xi)d\xi:\text{ }I\text{ is a closed
interval of the real line and }s\in I\right\}
\]
is the so-called maximal function associated with $k$. Since
 $k\in L^{1}(\mathbb{R},\mathbb{R})$, by
Hardy-Littlewood maximal theorem there exists a subset
$N(T,z)\subset\lbrack0,+\infty)$ of zero measure such that
${\sf Max}_{k}(s)<+\infty$ for every $s\in\lbrack0,T]\setminus
N(T,z).$ Let $N(z):=\cup_{m\in \N}N(m,z)$. Then for all
$s\in\lbrack
0,+\infty)\setminus N(z)$%
\begin{equation}\label{claima}
\sup_{n}\left\vert n(\varphi(z,s,s+\frac{1}{n})-z)\right\vert
<+\infty.
\end{equation}

Let $M:=N(0)\cup N(1/2).$ Clearly, $M$ is a  subset of
$[0,+\infty)$ of zero measure. We let
\[
\Gamma:[0,+\infty)\rightarrow2^{\mathrm{Hol}(\mathbb{D},\mathbb{C)}},\text{
}s\mapsto\Gamma(s)=\left\{
\begin{array}
[c]{ll}%
{\sf ac}(g_{n,s}) & s\notin M,\\
\{\id\} & s\in M,
\end{array}
\right.
\]
where
$g_{n,s}:=n(\varphi_{s,s+1/n}-\id)\in\mathrm{Hol}(\mathbb{D},\mathbb{C)}$
and ${\sf ac}(g_{n,s})$ denotes the accumulation points of the
sequence $\{g_{n,s}\}_{n}$ in the metric space
$\mathrm{Hol}(\mathbb{D},\mathbb{C)}$. The multifunction
$\Gamma$ is well-defined and, since
$\mathrm{Hol}(\mathbb{D},\mathbb{C)}$ is a metric space,
$\Gamma(s)$ is a closed subset of
$\mathrm{Hol}(\mathbb{D},\mathbb{C)}$ for every $s\geq0.$

Next step is to prove that $\Gamma(s)$ is  non-empty for all
$s\geq 0$. This is true by definition if $s\in M$. Thus, fix
$s\in [0,+\infty)\setminus M$. As recalled in section
\ref{preli}, $\varphi_{s,s+1/n}-\id$ belongs to
$\mathrm{Gen}(\mathbb{D})$ for all $n\in\mathbb{N}$. Moreover,
$\mathrm{Gen}(\mathbb{D})$ is a real cone in
$\mathrm{Hol}(\mathbb{D},\mathbb{C)}$, thus $\{g_{n,s}\}$ is a
sequence in $\mathrm{Gen}(\mathbb{D})$. By the very definition
of $M$,
\[
\max\{\sup_{n}|g_{n,s}(0)|,\sup_{n}|g_{n,s}(1/2)|\}<+\infty.
\]
Hence, we can apply Lemma \ref{Lema de los dos puntos} and
conclude that the sequence $\{g_{n,s}\}$ has accumulation
points in $\mathrm{Hol}(\mathbb{D},\mathbb{C)}$, so that
$\Gamma(s)$ is not empty. Thus $\Gamma$ satisfies hypothesis
(i) of Theorem~\ref{seleccion medible}.

In order to check condition (ii) in Theorem \ref{seleccion
medible} for $\Gamma$, we fix
$f\in\mathrm{Hol}(\mathbb{D},\mathbb{C)}$ and $r>0.$ Since $M$
has zero measure, we have only to prove that
\[
A_{f,r}=\{s\in\lbrack0,+\infty)\setminus M:\exists g\in {\sf
ac}(g_{n,s})\text{ with }d_{H}(f,g)<r\}
\]
is Lebesgue measurable, where $d_{H}$ is the canonical
Fr\'{e}chet distance defining the topology of
$\mathrm{Hol}(\mathbb{D},\mathbb{C)}$. Bearing in mind Lemma
\ref{Lema de los dos puntos} and the argument above, we see
that
\[
A_{f,r}:=%
{\textstyle\bigcup_{l=2}^{\infty}}
{\textstyle\bigcap_{n=1}^{\infty}}
{\textstyle\bigcup_{k=n}^{\infty}}
\{s\in\lbrack0,+\infty)\setminus M:\text{ }d_{H}(f,g_{k,s})<r\left(
1-\frac{1}{l}\right)  \}.
\]
Hence, it is enough to prove that, for every $k\in\mathbb{N}$, every $s\geq0$
and every $r^{\ast}>0,$ the subset
\[
B_{k,f,r^{\ast}}:=\{s\in\lbrack0,+\infty)\setminus M:\text{ }d_{H}%
(f,g_{k,s})<r^{\ast}\}
\]
is Lebesgue measurable. Since the functions $\lbrack0,+\infty
)\ni s
\mapsto\varphi_{s,s+1/k}\in\mathrm{Hol}(\mathbb{D},\mathbb{C)}$
are continuous (see Proposition \ref{EF-continuidad}), then
$\lbrack0,+\infty)\ni s\mapsto P_{k}(s):=g_{k,s}\in\mathrm{Hol}(\mathbb{D}%
,\mathbb{C)}$ is also continuous for every $k\in\mathbb{N}.$ Therefore, the
inverse image by $P_{k}$ of $B(f,r^{\ast})$ (the open ball in $\mathrm{Hol}%
(\mathbb{D},\mathbb{C)}$ with center $f$ and radius $r^{\ast})$
is an open subset of $\ [0,+\infty)$. Since
\[
B_{k,f,r^{\ast}}=P_{k}^{-1}(B(f,r^{\ast}))\setminus M,
\]
then $B_{k,f,r^{\ast}}$ is Lebesgue measurable.

Therefore, the multifunction $\Gamma$ satisfies the hypotheses
of Theorem \ref{seleccion medible}.  Thus there exists a
measurable selector $\sigma
:[0,+\infty)\rightarrow\mathrm{Hol}(\mathbb{D},\mathbb{C)}$ for
$\Gamma.$ We define
$G:\mathbb{D}\times\lbrack0,+\infty)\rightarrow\mathbb{C} $ by
\[
G(z,s):=\sigma\lbrack s](z),\quad\text{ for
}z\in\mathbb{D}\text{ and }s\geq0.
\]
Bearing in mind the definition of accumulation points in metric
spaces, we deduce that, for every
$s\in\lbrack0,+\infty)\setminus M,$ there exists a  strictly
increasing sequence $\{n_{k}(s)\}$ of natural numbers such
that, for all $z\in\mathbb{D},$
\begin{equation}\label{Glimit}
G(z,s):=\lim_{k\to \infty}n_{k}(s)(\varphi(z,s,s+1/n_{k}(s))-z)
\end{equation}
and the convergence is uniform on compacta of $\mathbb{D}.$ In
particular, because $\mathrm{Gen}(\mathbb{D})$ is a closed
subset of $\mathrm{Hol}(\mathbb{D},\mathbb{C)}$ (see
\cite[Consequence of Theorem 1.4.14]{Abate} or
\cite[p.76]{Shoikhet}), we see that $z\mapsto G(z,s)$ belongs
to $\mathrm{Gen}(\mathbb{D}),$ for every
$s\in\lbrack0,+\infty)\setminus M$. Moreover, bearing in mind
that $z\mapsto G(z,s)=z,$ for every $s\in M,$ we  deduce that
$z\mapsto G(z,s)$ belongs to $\mathrm{Gen}(\mathbb{D}),$ for
every $s\in\lbrack0,+\infty).$

\textit{Part (b): } According to Definition
\ref{Definicion-VF}, we have to check WHVF1, WHVF2 and WHVF3.
Fixing $z\in\mathbb{D},$ we see that by the very definition,
\[
\lbrack0,+\infty)\ni s\mapsto G(z,s)\in\mathbb{C}
\]
is the composition of the measurable selector $\sigma$ and the
 continuous functional of
$\mathrm{Hol}(\mathbb{D},\mathbb{C)}$ given by evaluation at
$z$. Thus,  WHVF1 holds. Also, WHVF2 holds trivially by the
very definition.

To prove property WHVF3, we argue as follows. Fix $z\in \D$ and
$T>0$. By EF3, there exists $k_z\in L^{d}([0,T+1],\mathbb{R})$
non-negative such that
\[
\left\vert \varphi(z,t,t+\frac{1}{n})-\varphi(z,t,t)\right\vert
\leq\int _{t}^{t+1/n}k_z(\xi)d\xi,
\]
for $0\leq t\leq T$ and every $n\in\mathbb{N}.$ The map
$s\mapsto\int_{s}^{s+1/n}k_z(\xi)d\xi$ is differentiable with
derivative $k(s)$ in $\lbrack0,T]$ outside a set $N_0(z,T)$ of
zero measure. Let $N(z,T):=M\cup N_{0}(z,T)$. Then  for every
$s\in\lbrack0,T]\setminus N(z,T),$
\[
\left\vert n_{k}(s)(\varphi(z,s,s+1/n_{k}(s))-z)\right\vert
\leq n_{k}(s)\int_{s}^{s+1/n_{k}(s)}k_z(\xi)d\xi.
\]
Taking limit in $k,$  by \eqref{Glimit}, we conclude that
\begin{equation}\label{Gtrivia}
|G(z,s)|\leq k_{z}(s)
\end{equation}
for almost every $s\in[0,T]$.

Now fix $r\in(0,1)$ and $T>0.$ By Part (a), we know that
$z\mapsto G(z,s)$ belongs to $\mathrm{Gen}(\mathbb{D}),$ for
every $s\in\lbrack0,+\infty)$. By \cite[Section
3.5]{Shoikhet}, there exist $a_{s}\in\mathbb{C}$ and $q_{s}\in\mathrm{Hol}%
(\mathbb{D},\mathbb{C})$ with $\Re q_{s}\geq0$ and
\[
G(z,s)=a_{s}-\overline{a_{s}}z^{2}-zq_{s}(z),\text{ }z\in\mathbb{D},\text{
}s\geq0.
\]
Since $G(0,s)=a_{s},$ equation \eqref{Gtrivia} provides a
function $k_{0}\in L^{d}([0,T],\mathbb{R})$ such that
\[
|a_{s}-\overline{a_{s}}z^{2}|\leq2k_{0}(t),\quad\text{ for
}s\in\lbrack0,T]\text{ and }|z|\leq r.
\]
Again by \eqref{Gtrivia}, we can find another function
$k_{1/2}\in L^{d}([0,T],\mathbb{R})$ such that
\[
|G(1/2,s)|\leq k_{1/2}(s),\text{ for }s\in\lbrack0,T].
\]
Therefore, for $s\in\lbrack0,T],$ we have%
\[
|q_{s}(1/2)|\leq2|a_{s}-\frac{1}{4}\overline{a_{s}}|+2|G(1/2,s)|\leq
3k_{0}(s)+2k_{1/2}(s).
\]
Since $s\mapsto G(z,s)$ is measurable for all fixed $z\in \D$,
it follows that both maps $s\mapsto a_{s}$ and $s\mapsto
q_{s}(1/2)$ belong to $L^{d}([0,T],\mathbb{C}).$ Now,  the
distortion theorem for Carath\'{e}odory functions \cite[pages
39-40]{Pommerenke} shows that, when $|z|\leq r$ and
$s\in\lbrack0,T]$
\begin{align*}
|G(z,s)|  & \leq2k_{0}(s)+|q_{s}(z)|\leq2k_{0}(s)+\frac{1+|z|}{1-|z|}%
|q_{s}(0)|\\
& \leq2k_{0}(s)+\frac{1+|z|}{1-|z|}\frac{1+1/2}{1-1/2}|q_{s}(1/2)|\\
& \leq2k_{0}(s)+\frac{1+r}{1-r}(9k_{0}(s)+6k_{1/2}(s)),
\end{align*}
showing WHVF3.

\textit{Part (c): }We have to prove that, given
$(z,s)\in\mathbb{D} \times\lbrack0,+\infty)$, the positive
trajectory of the weak holomorphic vector field $G$ with
initial data $(z,s)$ is exactly
\[
\lbrack s,+\infty)\ni t\rightarrow\varphi(z,s,t).
\]
Recall that by Proposition \ref{EF-continuidadabsoltuta}, this
function is absolutely continuous in $[s,+\infty)$ and
$\varphi(z,s,s)=z$. Thus we have only to show that  for almost
every $t\in(s,+\infty)$
\[
\frac{\partial\varphi}{\partial t}(z,s,t):=\lim_{h\longrightarrow0}%
\frac{\varphi(z,s,t+h)-\varphi(z,s,t)}{h}=G(\varphi(z,s,t),t).
\]
Let us fix $z\in\mathbb{D}$ and $s\geq0$. Let
$N_{1}(z,s)\subset\lbrack s,+\infty)$ be a set of zero measure
such that $\lbrack s,+\infty )\ni t\rightarrow\varphi(z,s,t)$
is differentiable  for every $t\in(s,+\infty)\setminus
N_{1}(z,s)$. Let $M$ be the set of zero measure defined in Part
(a). Then, for every $t\in(s,+\infty)\setminus(N_{1}(z,s)\cup
M)$,
\begin{align*}
\frac{\partial\varphi}{\partial t}(z,s,t)  & =\lim_{k}\frac{\varphi
(z,s,t+1/n_{k}(t))-\varphi(z,s,t)}{1/n_{k}(t)}\\
& =\lim_{k}n_{k}(t)\left(  \varphi(\varphi(z,s,t),t,t+1/n_{k}(t))-\varphi
(z,s,t)\right) \\
& =G(\varphi(z,s,t),t),
\end{align*}
and we are done.
\end{proof}

As a consequence of the previous results we have the following
interesting fact:

\begin{corollary}
\label{Univalencia} Let $(\varphi_{s,t})$ be an evolution
family of order $d\geq 1$ in the unit disc. Then, every
$\varphi_{s,t}$ is univalent.
\end{corollary}

\begin{proof}
By Theorem \ref{EF-implica-VF} the elements of the evolution
family $(\varphi_{s,t})$ are trajectories of a weak holomorphic
vector fields. By inequality (\ref{Eq_Univalencia}) they are
univalent in the unit disc.
\end{proof}

\begin{theorem}
\label{Derivacion independiente de z (en la t)} Let
$(\varphi_{s,t})$ be an evolution family of order $d\geq 1$ in
the unit disc.

\begin{enumerate}
\item For every $s\geq0,$ there exists a set $M(s)\subset\lbrack s,+\infty)$ (not depending on $z$) of
zero measure  such that, for every $t\in(s,+\infty )\setminus
M(s)$, the function
\[
\mathbb{D}\ni z\mapsto\frac{\partial\varphi}{\partial
t}(z,s,t)=\lim
_{h\rightarrow0}\frac{\varphi_{s,t+h}(z)-\varphi_{s,t}(z)}{h}\in\mathbb{C}%
\]
is a well-defined holomorphic function on $\mathbb{D}.$

\item Let $G:\mathbb{D}\times\lbrack0,+\infty)$ $\mathbb{\rightarrow C}$ be a
Herglotz vector field whose positive trajectories are
$(\varphi_{s,t})$. Fixed $s\geq0$. Then there exists a set
$M(s)\subset\lbrack s,+\infty)$ (not depending on $z$) of zero
measure  such that, for every
$t\in(s,+\infty)\setminus M(s)$ and every $z\in\mathbb{D}$, it holds that%
\begin{equation}\label{diffGev}
\frac{\partial\varphi}{\partial t}(z,s,t)=G(\varphi_{s,t}(z),t).
\end{equation}
\end{enumerate}
\end{theorem}

\begin{proof}
(1) Fix $s\geq0.$ By Proposition \ref{EF-continuidadabsoltuta}
 the map $\lbrack s,+\infty)\ni t
\mapsto \varphi(z,s,t)\in\mathbb{C}$ is absolutely continuous
in $[s,+\infty)$ for all fixed $z\in\mathbb{D}$. Thus there
exists a  set of zero measure $N(z,s)\subset\lbrack s,+\infty)$
such that, for every $t\in(s,+\infty)\setminus N(z,s)$ the
following limit exists
\[
D_{s,t}(z)=\frac{\partial\varphi}{\partial t}(z,s,t)=\lim_{h\rightarrow0}%
\frac{\varphi(z,s,t+h)-\varphi(z,s,t)}{h}.
\]
Now, define
\[
N(s):=%
{\textstyle\bigcup_{n=1}^{\infty}}
N(\frac{1}{n+1},s).\text{ }%
\]
The set $N(s)$ has zero measure and it is independent of $z$.
We are going to show that this is the subset we are looking
for; namely
$\lim_{h\rightarrow0}(\varphi(z,s,t+h)-\varphi(z,s,t))/h$
exists  for all $t\in(s,+\infty)\setminus N(s)$ uniformly on
compacta of $\mathbb{D}$.

First of all we show that for every $t\in(s,+\infty)\setminus
N(s),$ the family
\[
\mathcal{F}_{s,t}:=\{\frac{1}{h}(\varphi_{s,t+h}-\varphi_{s,t}):0<|h|<t-s\}
\]
is relatively compact in $\mathrm{Hol}(\mathbb{D},\mathbb{C})$.
To this aim, we will work separately two cases: $(a)$
$0<h<t-s;$ $(b)$ $s-t<h<0.$

Case $(a)$:$\,$\ Since $h>0$ and by EF2,
\[
\mathcal{F}_{s,t}=\left\{
f_{h}\circ\varphi_{s,t}:0<h<t-s\right\}  ,
\]
where
$f_{h}:=\frac{1}{h}(\varphi_{t,t+h}-\id)\in\mathrm{Hol}(\mathbb{D}
,\mathbb{C})$. Since $\varphi_{s,t}$ is holomorphic in
$\mathbb{D},$ and by Montel's theorem,  we only need to check
that
\[
\mathcal{F}_{s,t}^{\ast}:=\left\{  f_{h}:0<h<t-s\right\}
\]
is a bounded subset of $\mathrm{Hol}(\mathbb{D},\mathbb{C})$.
 Assume this is not the case.
Then there exists a sequence $\{f_{n}\}$ (with $f_n
:=f_{h_{n}}$) in $\mathcal{F}_{s,t}^{\ast}$ and $r\in(0,1)$
such that
\begin{equation}\label{star}
\lim_{n\to \infty} \max\left\{ |f_{n}(z)|:|z|\leq r\right\}
=+\infty.
\end{equation}
Since the sequence $\{\varphi_{t,t+h_{n}}-\id\}_{n}$ belongs to
$H^{\infty}(\mathbb{D}),$ we may assume that $\lim_{n}h_{n}=0$.
Moreover, letting $z_{1}:=\varphi(1/2,s,t)$ and
$z_{2}:=\varphi(1/3,s,t)$, then $z_1\neq z_2$ because
$\varphi_{s,t}$ is univalent (see Corollary \ref{Univalencia}).
Hence, since $h_{n}>0$ and $t\not\in N(1/2,s)$
\begin{align*}
D_{s,t}(1/2)  & =\lim_{n}\frac{\varphi(1/2,s,t+h_{n})-\varphi(1/2,s,t)}{h_{n}%
}\\
& =\lim_{n}\frac{\varphi(\varphi(1/2,s,t),t,t+h_{n})-\varphi(1/2,s,t)}{h_{n}%
}\\
&
=\lim_{n}\frac{\varphi(z_{1},t,t+h_{n})-z_{1}}{h_{n}}=\lim_{n}f_{n}(z_{1}).
\end{align*}
Similarly, one can check the existence of the limit $\lim_{n}
f_{n}(z_{2})$. Now,  note that $\mathcal{F}_{s,t}^{\ast}\subset
\mathrm{Gen}(\D)$ since $h>0$. Therefore, we can apply Lemma
\ref{Lema de los dos puntos} to the sequence $\{f_{n}\}$ and
the two points $z_{0},z_{1}$, contradicting \eqref{star}.

Case $(b)$: the proof of this case is similar to that of Case
$(a)$ and we only sketch it. Since $h<0$ and by EF2, we see
that $\mathcal{F}_{s,t}:=\left\{
f_{h}\circ\varphi_{s,t+h}:s-t<h<0\right\}$, where
$f_{h}:=-\frac{1}{h}(\varphi_{t+h,t}-\id)\in\mathrm{Hol}(\mathbb{D}
,\mathbb{C})$. By Proposition \ref{EF-continuidad} and Montel's
theorem, we only have to check that
$\mathcal{F}_{s,t}^{\ast}:=\left\{  f_{h}:s-t<h<0\right\}$ is a
bounded subset of $\mathrm{Hol}(\mathbb{D},\mathbb{C})$. Again,
we argue by contradiction assuming the existence of a  sequence
$\{f_{n}\}\subset\mathcal{F}_{s,t}^{\ast}$ which is not bounded
on some compact subset of $\D$. This time we define
$z_{1,n}:=\varphi(1/2,s,t+h_{n})$, $
z_{2,n}:=\varphi(1/3,s,t+h_{n})$. Because  $\varphi_{s,t}$ is
univalent, we find that $\lim_{n}z_{1,n}\neq \lim_{n}z_{2,n}$
and
\begin{align*}
D_{s,t}(1/2)  & =\lim_{n}\frac{\varphi(1/2,s,t+h_{n})-\varphi(1/2,s,t)}{h_{n}%
}\\
& =\lim_{n}\frac{\varphi(1/2,s,t+h_{n})-\varphi(\varphi(1/2,s,t+h_{n}%
),t+h_{n},t)}{h_{n}}\\
& =\lim_{n}\frac{z_{1,n}-\varphi(z_{1,n},t+h_{n},t)}{h_{n}}=\lim_{n}%
f_{n}(z_{1,n}).
\end{align*}
In a similar way, it can be checked the existence of the limit
$\lim_{n} f_{n}(z_{2,n})$. Again, this forces a contradiction
to Lemma \ref{Lema de los dos puntos}.

Thus the family $\mathcal{F}_{s,t}$ is  relatively compact in
$\mathrm{Hol}(\mathbb{D},\mathbb{C})$. Let $\psi, \phi$ be two
of its limits. By the very definition of $N(s)$
\[
D_{s,t}(\frac{1}{m+1})=\psi(\frac{1}{m+1})=\phi(\frac{1}{m+1}),
\]
for every $m\in\mathbb{N}$. But $\{\frac{1}{m+1}\}$ is a
sequence accumulating at $0$, hence by the identity principle
$\psi=\phi$. This shows that
\[
\lim_{h\rightarrow0}\frac{\varphi(z,s,t+h)-\varphi(z,s,t)}{h},
\]
exists, for all $t\in(s,+\infty)\setminus N(s)$ uniformly on
compacta of $\mathbb{D}$, ending the proof of (1).

(2) Fix $s\geq0.$ By  part (1), there exists a set
$N_{0}(s)\subset\lbrack s,+\infty)$ of zero measure (not
depending on $z$) such that, for every
$t\in(s,+\infty)\setminus N_{0}(s)$, the function
$\mathbb{D}\ni z\mapsto\frac{\partial\varphi}{\partial
t}(z,s,t)$ is a well defined holomorphic function on $\D$. Let
$\{z_n\}$ be any sequence converging to $0$. Then for all $n\in
\N$ there exists a set of zero measure $N(z_n, s)$ such that
\begin{equation}\label{equaln}
\frac{\partial\varphi}{\partial
t}(z_{n},s,t)=G(\varphi(z_{n},s,t),t)
\end{equation}
for all $t\in [s,+\infty)\setminus N(z_n,s)$. Let
$N(s)=N_0(s)\cup \bigcup_{n} N(z_n,s)$. Then $N(s)$ has measure
zero and for all $t\in [s,+\infty)\setminus N(s)$ equation
 \eqref{equaln} holds. By the identity principle for
 holomorphic maps, the two holomorphic functions $\frac{\partial\varphi}{\partial
t}(\cdot,s,t)$ and $G(\varphi(\cdot,s,t),t)$ are then equal on
$\D$, proving (2).
\end{proof}

\begin{corollary}\label{uniqueHerglotz}
If $G, \tilde{G}$ are Herglotz vector fields with the same
positive trajectories then $G(z,t)=\tilde{G}(z,t)$ for almost
every $t\in [0,+\infty)$ and all $z\in \D$.
\end{corollary}

\begin{proof}
By Theorem \ref{Herglotz-implica-EF} the positive trajectories
of $G$ and $\tilde{G}$ are evolution families of the unit disc.
In particular they are univalent by Corollary
\ref{Univalencia}. The claim follows then from~\eqref{diffGev}.
\end{proof}

The next result studies the dependence of evolution families
with respect to the ``$s$ variable''.

\begin{theorem}
\label{Derivacion indep. de z (en la s)} Let $(\varphi_{s,t})$
be an evolution family of order $d\geq 1$ in the unit disc.

\begin{enumerate}
\item For every $t>0,$ there exists a set $N(t)\subset\lbrack0,t]$ (not depending on $z$) of zero measure
such that, for every $s\in(0,t)\setminus N(t)$, the function
\[
\mathbb{D}\ni z\mapsto\frac{\partial\varphi}{\partial
s}(z,s,t):=\lim
_{h\rightarrow0}\frac{\varphi_{s+h,t}(z)-\varphi_{s,t}(z)}{h}\in\mathbb{C}%
\]
is a well-defined holomorphic function on $\mathbb{D}.$

\item Let $G:\mathbb{D}\times\lbrack0,+\infty)$ $\mathbb{\rightarrow C}$ be a
Herglotz vector field whose positive trajectories are
$(\varphi_{s,t})$. Fix $t>0.$ Then, there exists a set
$N(t)\subset\lbrack0,t]$ (not depending on $z$) of zero measure
such that, for every $s\in (0,t)\setminus N(t)$ and every
$z\in\mathbb{D}$
\begin{equation}\label{sGev}
\frac{\partial\varphi}{\partial s}(z,s,t)=-G(z,s)\varphi_{s,t}^{\prime}(z).
\end{equation}
\end{enumerate}
\end{theorem}

\begin{proof}
(1) Fix $t>0.$ By Proposition \ref{EF-continuidadabsoltuta}
 the map $\lbrack0,t]\ni s
\mapsto \varphi(z,s,t)\in\mathbb{C}$ is absolutely continuous
in in $[0,t]$, for all fixed $z\in\mathbb{D}$. Thus there
exists a  set of zero measure $N_{1}(z,t)\subset\lbrack0,t]$
such that, for every $s\in (0,t)\setminus N_{1}(z,t),$ the
following limit exists
\[
D_{s,t}(z)=\frac{\partial\varphi}{\partial s}(z,s,t)=\lim_{h\rightarrow0}%
\frac{\varphi(z,s+h,t)-\varphi(z,s,t)}{h}.
\]
Moreover, let $\{r_{n}\}\subset(0,1)$ be a sequence converging
to $1$. For $R>0$, let $\D(R)=\{\zeta\in \C: |\zeta|<R\}$. By
Lemma \ref{EF-acotacion}, for all $n\in \N$ there exists
$R_{n}:=R(r_{n},t)\in(0,1)$ such that
\[
A_{n}:=\{\varphi(z,u,v):|z|\leq r_{n},\text{ }0\leq u\leq v\leq
t+1\}\subset \D(R_{n}).
\]
 Let $G:\mathbb{D}\times\lbrack0,+\infty)$ $\mathbb{\rightarrow C}$ be a
Herglotz vector field whose positive trajectories are
$(\varphi_{s,t})$ (such a vector field exists by Theorem
\ref{EF-implica-VF}). Let $k_{n}:=k_{R_{n},t}\in
L^{d}([0,t+1],\mathbb{R})$ be the non negative function given
by property WHVF3 in Definition \ref{Definicion-VF}. There
exists a  set $N_{2}(n,t)\subset\lbrack0,t]$ of zero measure
such that, for every $s\in(0,t)\setminus N_{2}(z,t)$
\[
k_{n}(s)=\lim_{h\rightarrow0}\frac{1}{h}\int_{s}^{s+h}k_{n}(\eta)d\eta.
\]
Let us define
\[
N(t):=\left(
{\textstyle\bigcup_{n=1}^{\infty}}
N_{1}(\frac{1}{n+1},t)\right)  \cup\left(
{\textstyle\bigcup_{n=1}^{\infty}}
N_{2}(n,t)\right)  .
\]
Obviously, $N(t)$ is a subset of $[0,t]$ of zero measure,
independent of $z$. We are going to prove that for all
$s\in(0,t)\setminus N(t)$  the following limit
\[
\lim_{h\rightarrow0}\frac{\varphi(z,s+h,t)-\varphi(z,s,t)}{h}
\]
exists  uniformly on compacta of $\mathbb{D}.$

First of all we show that for every $s\in(0,t)\setminus N(t)$
the family
\[
\mathcal{F}_{s,t}:=\{F_{h}:=\frac{1}{h}(\varphi_{s+h,t}-\varphi_{s,t}%
):0<h<t-s\text{ or }-s<h<0\}
\]
is a relatively compact in
$\mathrm{Hol}(\mathbb{D},\mathbb{C})$. To this aim, we consider
two  cases: $(a)$ $0<h<t-s;$ $(b)$ $-s<h<0.$

Case $(a)$: \ Fix $r\in(0,1).$ Let $n\in \N$ be such that
$r_{n}>r$, and let $\rho_{n}\in(0,1)$ be such that
$\rho_{n}>R_{n}$. Set $z_{h}:=\varphi(z,s,s+h)$. Then, for
every $|z|\leq r,$ the point $z_h\in A_{n}$ and
\begin{align*}
|F_{h}(z)|  & =\left\vert \frac{1}{h}\left(  \varphi(z,s+h,t)-\varphi
(\varphi(z,s,s+h),s+h,t)\right)  \right\vert \\
& =\frac{1}{h}\left\vert \frac{1}{2\pi i}\int_{C^{+}(0,\rho_{n})}\varphi
(\xi,s+h,t)\left(  \frac{1}{\xi-z}-\frac{1}{\xi-z_{h}}\right)  d\xi\right\vert
\\
& \leq\frac{1}{h}\rho_{n}\frac{|z-z_{h}|}{(\rho_{n}-r_{n})(\rho_{n}-R_{n})}.
\end{align*}
Setting $C:=C(r,t)=\dfrac{\rho_{n}}{(\rho_{n}-r_{n})(\rho_{n}-R_{n}%
)}>0$ and recalling the definition of $A_{n},$ we have that
there exists $\tilde{C}>0$ such that
\begin{align*}
|F_{h}(z)|  & \leq C\frac{1}{h}\left\vert \varphi(z,s,s+h)-z\right\vert
=C\frac{1}{h}\left\vert \int_{s}^{s+h}G(\varphi(z,s,\xi),\xi)d\xi\right\vert
\\
& \leq C\frac{1}{h}\int_{s}^{s+h}k_{n}(\xi)d\xi\leq
\tilde{C}<+\infty,
\end{align*}
where the last inequality follows from $s\notin N_{2}(n,t)$.
Hence, $\sup\{|F_{h}(z)|:|z|\leq r,$ $0<h<t-s\}<+\infty$ as
wanted.

Case $(b)$:  the proof is similar to that of  case $(a)$ and we
omit it .

Now, arguing as in the last part of the proof of part (1) of
Theorem \ref{Derivacion independiente de z (en la t)} we can
see that
$\lim_{h\rightarrow0}(\varphi(z,s+h,t)-\varphi(z,s,t))/h$
exists for all $s\in(0,t)\setminus N(t)$  uniformly on compacta
of $\mathbb{D}$, concluding the proof of (1).

(2) Fix $t>0.$ Let $N_{1}\subset\lbrack0,+\infty)$ be the set
of zero measure given by Theorem \ref{Derivacion independiente
de z (en la t)}.(1) such that $\frac{\partial\varphi}{\partial
t}(z,0,u)=G(\varphi(z,0,u),u)$ for all
$u\in(0,+\infty)\setminus N_{1}$ and for all $z\in \D$. Let
$N_{2}:=N_{2}(t)\subset\lbrack0,t]$ be the set of zero measure
prescribed by part (1) of this theorem.

Let $N:=N_{1}\cup N_{2}$. Differentiating with respect to $u$
the identity $\varphi(z,0,t)=\varphi(\varphi(z,0,u),u,t)$, for
 $z\in\mathbb{D}$ and $u\in(0,t)\setminus N$ we obtain
\begin{align*}
0  & =\varphi^{\prime}(\varphi(z,0,u),u,t)\frac{\partial\varphi}{\partial
u}(z,0,u)+\frac{\de \v}{\de u}(\varphi(z,0,u),u,t)\\
&
=\varphi^{\prime}(\varphi(z,0,u),u,t)G(\varphi(z,0,u),u)+\frac{\de
\v}{\de u}(\varphi(z,0,u),u,t).
\end{align*}
Therefore $\varphi^{\prime}(w,u,t)G(w,u)=-\frac{\de \v}{\de
u}(w,u,t)$ for all $w=\varphi(z,0,u)$. Since the
$\varphi_{0,u}$'s are univalent, the identity principle for
holomorphic maps implies the result.
\end{proof}

 Now we are going to show that the $\tau_s$ appearing in the
Berkson-Porta type decomposition formula for Herglotz vector
fields are related to  Denjoy-Wolff points of the elements of
the associated evolution families as in the classical
Berkson-Porta formula for semigroups:

\begin{theorem}
Let $(\varphi_{s,t})$ be an evolution family of order $d\geq 1$
in the unit disc, let $G(z,t)$ be the Herglotz vector field of
order $d\geq 1$ which solves \eqref{main-eq} and let
\[
G(z,s)=(z-\tau_{s})(\overline{\tau_{s}}z-1)p(z,s),\text{ }
z\in\mathbb{D},\text{ }s\geq0,
\]
be its Berkson-Porta type decomposition
\eqref{Herglotz-vf-main}. Let $Z:=\{s\in [0,+\infty): G(\cdot,
s)\not\equiv 0\}$. Then for almost every $s\in Z$ there exists
a decreasing sequence $\{t_{n}(s)\}$ converging to $s$ such
that $\varphi_{s,t_{n}(s)}\not\equiv \id_\D$ and, denoting by
$\tau(s,n)$ the Denjoy-Wolff point of $\varphi_{s,t_{n}(s)}$,
it holds
\[
\tau_{s}=\lim_{n\to\infty}\tau(s,n).
\]
\end{theorem}

\begin{proof} By \eqref{Glimit} there exists a set of zero measure $M\subset
\lbrack0,+\infty),$ such that for every
$s\in(0,+\infty)\setminus M$ there exists a strictly increasing
sequence of natural numbers $\{n_{k}(s)\}$ such that, defining
\[
f_k(z,s):=n_{k}(s)(\varphi(z,s,s+1/n_{k}(s))-z),
\]
it follows that  $G(z,s)$ is the uniform limit on compacta of
$\D$ of the sequence $\{f_k(z,s)\}$. Note that by the classical
Berkson-Porta formula,  $G(\cdot,s)\in \hbox{Gen}(\D)$ for
$s\geq 0$ fixed and also  $f_k(\cdot,s)\in \hbox{Gen}(\D)$ for
$s\geq 0$ fixed and all $k\geq 0$ (see Section 2).

Fix $s\in Z\setminus M$. Therefore there exists $m(s)\in \N$
such that  $\varphi(\cdot ,s,s+1/n_{k}(s))\not\equiv\id_\D$ and
$f_{k}(\cdot,s)\not\equiv 0$ for $k\geq m(s)$.

We claim that $\{s+1/n_{k}(s)\}_{k\geq m(s)}$ is the sequence
(which we relabel $\{t_n(s)\}$) we are looking for. Let
$\tau(s,k)$ be the  Denjoy-Wolff point of
$\varphi_{s,s+1/n_{k}(s)}$.

We claim that $BP_{\tau}(f_{k}(\cdot,s))=\tau(s,k),$ for all
$k$. Once this is proved then the result follows at once from
Proposition \ref{BP-continuidad}.

In case $\tau(s,k)\in\mathbb{D}$ then clearly $f_k(\tau(s,k),
s)=0$ and hence $BP_{\tau}(f_{k}(\cdot,s))=\tau(s,k)$, as
wanted.

In case $\tau (s,k)\in\partial\mathbb{D}$, then
$\angle\lim_{z\rightarrow\tau(s,k)}f_{k}(z,s)=0$, so
$\tau(s,k)$ is a boundary critical point for the generator
$f_{k}(\cdot,s)$ (see \cite{Contreras-Diaz-Pommerenke:Scand}
for further details about critical points). Bearing in mind
that $\tau(s,k)\in\partial\mathbb{D}$ is the Denjoy-Wolff point
of $\varphi_{s,s+1/n_{k}(s)},$ we have that
\[
\angle\lim_{z\rightarrow\tau(s,k)}\varphi^{\prime}(z,s,s+1/n_{k}(s))\in(0,1].
\]
Hence
\[
\angle\lim_{z\rightarrow\tau(s,k)}f_{k}^{\prime}(z,s)=n_{k}(s)\left(
\angle\lim_{z\rightarrow\tau(s,k)}\varphi^{\prime}(z,s,s+1/n_{k}(s))-1\right)
\in\lbrack0,+\infty).
\]
According to \cite{Contreras-Diaz-Pommerenke:Scand}, this
implies that $BP_{\tau}(f_{k}(\cdot,s))=\tau(s,k)$, as needed.
\end{proof}

\section{Evolution families with a common fixed point}

\begin{theorem}
\label{Continuidad-absoluta-multiplicador} Let
$(\varphi_{s,t})$ be an evolution family of order $d\geq 1$ of
the unit disc with Berkson-Porta data $(p,\tau)$. Suppose that
$\tau(t)\equiv \tau\in \oD$ is constant. Then there exists a
unique locally absolutely continuous function
$\lambda:[0,+\infty )\rightarrow\mathbb{C}$ with $\lambda'\in
L_{loc}^{d}([0,+\infty),\mathbb{C})$, $\lambda(0)=0$ and $\Re
\lambda(t)\geq\Re \lambda(s)\geq0$ for all $0\leq s\leq
t<+\infty$  such that for all $s\leq t$
\[
\varphi_{s,t}^{\prime}(\tau)=\exp(\lambda (s)-\lambda(t)).
\]
Moreover, if $\tau\in\mathbb{D}$, then
\[
\lambda(t)=(1-|\tau|^{2})\int_{0}^{t}p(\tau,\xi)d\xi\quad\text{for
all }t\geq0,
\]
while, if $\tau\in\partial\mathbb{D}$, then
\[
\lambda(t)=\int_{0}^{t}\left(  \angle\lim_{z\rightarrow\tau}\frac
{2|\tau-z|^{2}p(\frac{\tau+z}{\tau-z},\xi)}{1-|z|^{2}}\right)  d\xi
\quad\text{for all }t\geq0.
\]
\end{theorem}

\begin{proof}
\textit{Case  $\tau\in\mathbb{D}$}. Firstly, we assume that
$\tau=0.$ We write $\lambda(t)=\int_{0}^{t}p(0,\xi)d\xi$ for
all $t\geq0$. Fixed $s<t$ we have only to prove that
$\varphi_{s,t}^{\prime}(0)=\exp(\lambda(s)-\lambda (t)).$

\textbf{Claim 1. }\textit{For all $z\in\mathbb{D}$ and for all
$0\leq s<t,$ the function  $\lbrack s,t]\ni\xi\mapsto
p(\varphi_{s,\xi}(z),\xi)$ belongs to
$L^{d}([s,t],\mathbb{C}).$ Moreover,}
\begin{equation}\label{formav}
\varphi_{s,t}(z)=z\exp\left(
-\int_{s}^{t}p(\varphi_{s,\xi}(z),\xi )d\xi\right) .
\end{equation}

Assuming the claim,  since $\v'_{s,t}(0)=\lim_{z\to
0}\frac{\varphi_{s,t}(z)}{z}$, by \eqref{formav} we are left to
prove  that
\begin{equation}\label{clancy}
\lim_{z\rightarrow0}\int_{s}^{t}p(\varphi_{s,\xi}(z),\xi)d\xi=\int_{s}%
^{t}p(0,\xi)d\xi.
\end{equation}
Fix $\xi$. If $\Re p(z,\xi)=0$ for some $z\in \D$ then
$p(\cdot,\xi)\equiv i a_\xi$ for some $a_\xi\in \R$. If $\Re
p(\cdot,\xi)>0$, the holomorphic map $\D\ni
z\mapsto\frac{p(z,\xi)-p(0,\xi)}{p(z,\xi)+\overline{p(0,\xi)}}
$ sends the unit disc  into itself and fixes the point zero.
Then
\begin{align*}
\left\vert p(z,\xi)-p(0,\xi)\right\vert  & \leq|z|\left\vert p(z,\xi
)+\overline{p(0,\xi)}\right\vert \leq|z||p(z,\xi)|+|z||p(0,\xi)|\\
& \leq|z|\frac{1+|z|}{1-|z|}|p(0,\xi)|+|z||p(0,\xi)|=\frac{2\left\vert
z\right\vert }{1-\left\vert z\right\vert }\left\vert p(0,\xi)\right\vert ,
\end{align*}
where in the last inequality we have used \cite[pages 39-40]{Pommerenke}.
Therefore,
\[
\left\vert p(\varphi_{s,\xi}(z),\xi)-p(0,\xi)\right\vert \leq\frac{2\left\vert
\varphi_{s,\xi}(z)\right\vert }{1-\left\vert \varphi_{s,\xi}(z)\right\vert
}\left\vert p(0,\xi)\right\vert \leq\frac{2\left\vert z\right\vert
}{1-\left\vert z\right\vert }\left\vert p(0,\xi)\right\vert .
\]
Since the function $\lbrack s,t]\ni \xi\mapsto p(0,\xi)$
belongs to $L^{d}([s,t],\mathbb{C}),$ we have
\[
\left\vert \int_{s}^{t}p(\varphi_{s,\xi}(z),\xi)d\xi-\int_{s}^{t}p(0,\xi
)d\xi\right\vert \leq\frac{2\left\vert z\right\vert }{1-\left\vert
z\right\vert }\int_{s}^{t}\left\vert p(0,\xi)\right\vert d\xi
\]
and \eqref{clancy} follows from the dominated convergence
theorem.

In case $\tau\in\D\setminus\{0\}$, we conjugate $(\v_{s,t})$
with the automorphism
$\psi:=\frac{\tau-z}{1-\overline{\tau}z}$. The evolution family
$(\psi_{s,t})$ (defined by $\psi_{s,t}:=\psi\circ \v_{s,t}\circ
\psi$) has Berkson-Porta data $((1-|\tau|^2)p(\psi(z),t),0)$.
Since $\varphi_{s,t}^{\prime}(\tau)=\psi_{s,t}^{\prime}(0)$,
the result follows from the previous case.

\textit{Case } $\tau\in\partial\mathbb{D}.$ Conjugating with
the Cayley transform $T_\tau:z\mapsto\frac{\tau+z}{\tau-z}$, we
define a  family $\phi_{s,t}:=T_{\tau}\circ\varphi_{s,t}\circ
T_{\tau}^{-1}$ of holomorphic self maps of $\Ha:=\{w\in \C: \Re
w>0\}$. For $w\in \Ha$ and $t\in[0,+\infty)$ we let
$P(w,t):=2p(T_{\tau }^{-1}(w),t)$.

\textbf{Claim 2. }\textit{For all $w\in\mathbb{H}$ and for all
$0\leq s<t$ the function $\lbrack s,t]\ni \xi \mapsto\frac{\Re
P(\phi_{s,\xi}(w),\xi )}{\Re \phi_{s,\xi}(w)} \in
L^{d}([s,t],\mathbb{C}).$ Moreover,
\begin{equation}\label{claimRe}
\Re \phi_{s,t}(w)=\Re w\exp\left(  \int_{s}^{t}%
\frac{\Re P(\phi_{s,\xi}(w),\xi)}{\Re \phi_{s,\xi
}(w)}d\xi\right)  \text{ \quad for all }w\in\mathbb{H}
\end{equation}
and $\infty$ is the Denjoy-Wolff point of $\phi_{s,t}.$}

We  assume that $\Re p(w,\xi)>0$ for all $w$ and for all
$\xi\geq0$ (leaving to the reader the obvious modifications in
case $\Re p(\cdot,\xi)\equiv a_\xi i$, $a_\xi\in \R$ for some
$\xi$). By the Julia-Wolff-Carath\'{e}odory theorem (see, {\sl
e.g.}, \cite{Abate}), the number
\[
\widehat{\lambda}(\xi):=\inf\left\{ \frac{\Re P(w,\xi )}{\Re
w}:w\in\mathbb{H}\right\} =\angle\lim_{w\rightarrow
\infty}\frac{P(w,\xi)}{w}=\angle\lim_{w\rightarrow\infty}\frac
{\Re P(w,\xi)}{\Re w}%
\]
is well-defined for all $\xi$. Moreover, the function
$\xi\in\lbrack0,+\infty)\mapsto\widehat{\lambda}(\xi)$ is
measurable since
$\widehat{\lambda}(\xi)=\lim_{n\rightarrow\infty}
\frac{P(n,\xi)}{n}$ and $\xi\mapsto P(w,\xi)$ is measurable for
all $w\in \Ha$. In addition, since $\lbrack0,+\infty)\ni \xi
\mapsto \Re P(1,\xi)$ belongs to $L_{loc}^{d}([0,+\infty))$ and
$0\leq\widehat{\lambda}(\xi)\leq\Re P(1,\xi)$, we conclude that
the function $\lbrack0,+\infty)\ni\xi\mapsto\widehat{\lambda
}(\xi)$ also belongs to $L_{loc}^{d}([0,+\infty))$.

By \eqref{claimRe}
\begin{equation*}
\varphi_{s,t}^{\prime}(\tau)^{-1}   =\lim_{n\rightarrow+\infty}
\frac{\Re \phi_{s,t}(n)}{n} =\exp\left(  \lim_{n\rightarrow+\infty}\int_{s}^{t}\frac{\Re %
P(\phi_{s,\xi}(n),\xi)}{\Re \phi_{s,\xi}(n)}d\xi\right).
\end{equation*}
Now, for all fixed $n\in \N$
\[
\int_{s}^{t}\widehat{\lambda}(\xi)d\xi=\int_{s}^{t}\inf\left\{
\frac {\Re P(w,\xi)}{\Re w}:w\in\mathbb{H}\right\}
d\xi\leq\int_{s}^{t}\frac{\Re P(\phi_{s,\xi}(n),\xi )}{\Re
\phi_{s,\xi}(n)}d\xi.
\]
Thus, $\int_{s}^{t}\widehat{\lambda}(\xi)d\xi\leq\lim_{n\rightarrow+\infty
}\int_{s}^{t}\frac{\Re P(\phi_{s,\xi}(n),\xi)}{\Re %
\phi_{s,\xi}(n)}d\xi.$ Moreover, since $\infty$ is the
Denjoy-Wolff point of the function $\phi_{s,t}$ we have that
$\Re \phi_{s,t}(n)\geq n$ and, by \cite[(3.2)]{Pomm79} and
\cite[pages 39-40]{Pommerenke} we have
\[
\frac{\Re P(\phi_{s,\xi}(n),\xi)}{\Re \phi_{s,\xi
}(n)}\leq\frac{\Re P(n,\xi)}{n}\leq 4\Re P(1,\xi).
\]
Then, the sequence of measurable functions $\lbrack
s,t]\ni\xi\mapsto \frac{\Re P(n,\xi)}{n}$ is uniformly bounded
by a $L_{loc}^{d}([0,+\infty))$-function and converges
pointwise to $\widehat{\lambda}.$ Thus, the dominated
convergence theorem shows that
\[
\lim_{n}\int_{s}^{t}\frac{\Re P(\phi_{s,\xi}(n),\xi
)}{\Re \phi_{s,\xi}(n)}d\xi\leq\lim_{n}\int_{s}^{t}%
\frac{\Re P(n,\xi)}{n}d\xi=\int_{s}^{t}\left(  \lim_{n}%
\frac{\Re P(n,\xi)}{n}\right)
d\xi=\int_{s}^{t}\widehat{\lambda }(\xi)d\xi.
\]
Summing up, we have  $\varphi_{s,t}^{\prime}(\tau)=\exp\left(
-\int _{s}^{t}\widehat{\lambda}(\xi)d\xi\right)$ as wanted.

Now, we are left to prove Claim 1 and Claim 2.

\textbf{Proof of Claim 1.} Fix $z$ and $0\leq s<t<+\infty.$
Since the function $\xi\mapsto p(z,\xi)$ is measurable and
$\xi\mapsto\varphi_{s,\xi}(z)$ is continuous by Proposition
\ref{EF-continuidad}, it follows that the function $\lbrack
s,t]\ni\xi\mapsto p(\varphi_{s,\xi}(z),\xi)$ is measurable.
Moreover, for all $\xi,$ by \cite[pages 39-40]{Pommerenke}, we
have
\[
|p(\varphi_{s,\xi}(z),\xi)|\leq\frac{1+|\varphi_{s,\xi}(z)|}{1-|\varphi
_{s,\xi}(z)|}|p(0,\xi)|\leq\frac{2}{1-|z|}|p(0,\xi)|.
\]
Therefore, the map $\lbrack s,t]\ni\xi\mapsto
p(\varphi_{s,\xi}(z),\xi)\in L^{d}([s,t],\mathbb{C})$. Hence,
the function $\phi(u):=z\exp\left(
-\int_{s}^{u}p(\varphi_{s,\xi}(z),\xi)d\xi\right)  $ is
absolutely continuous in $[s,t]$ and
$\phi^{\prime}(u)=-\phi(u)p(\varphi_{s,u}(z),u)$. Assume that
$z\neq0.$ By \eqref{diffGev} and \eqref{Herglotz-vf}, recalling
that $\tau\equiv 0$, it follows
$\frac{\partial\varphi_{s,u}(z)}{\partial
u}=-\varphi_{s,u}(z)p(\varphi_{s,u}(z),u)$ almost everywhere,
thus
\[
\frac{\partial}{\partial u}\left(
\frac{\phi(u)}{\varphi_{s,u}(z)}\right)  \equiv 0
\]
for almost every $u\in\lbrack s,t]$ (notice that since
$\varphi_{s,u}$ is univalent then $\varphi_{s,u}^{-1}(0)=0$ and
the above quotient is well-defined for $z\neq 0$). Therefore,
there exists $c$ such that $\varphi_{s,u}(z)=c\phi(u)$ for all
$u.$ But, $\varphi_{s,s}(z)=z$ and $\phi(s)=z.$ Hence, $c=1$
and the claim is proved.

\textbf{Proof of Claim 2.} A direct computation shows that
equation \eqref{diffGev} translates to $\Ha$ in
\begin{equation}\label{diffP}
\frac{\partial\phi_{s,t}(w)}{\partial t}=P(\phi_{s,t}(w),t),
\end{equation}
which holds for almost every $t$ and every $w\in \Ha$.

Fix $w$ and $0\leq s<t<+\infty.$ Since the function $\xi\mapsto
P(w,\xi)$ is measurable and $\xi\mapsto\phi_{s,\xi}(z)$ is
continuous by Proposition \ref{EF-continuidad}, the function
$\lbrack s,t]\ni\xi\mapsto P(\phi_{s,\xi}(z),\xi)$ is
measurable. Moreover by the distortion theorem for
Carath\'eodory functions (see \cite{Pomm79}), for each
$\xi\in\lbrack s,t],$
\begin{equation}\label{distorP}
\frac{\Re P(\phi_{s,\xi}(w),\xi)}{\Re \phi_{s,\xi
}(w)}\leq\frac{|\phi_{s,\xi}(w)+1|^{2}}{(\Re \phi_{s,\xi}
(w))^{2}}\Re P(1,\xi).
\end{equation}
Fix a compact set $K$ in $\mathbb{H}.$ By Lemma
\ref{EF-acotacion}, the set $\{\phi_{s,\xi}(w):w\in
K,\xi\in\lbrack s,t]\}$ is  compact in $\mathbb{H}.$ Since the
function $\lbrack s,t]\ni\xi\mapsto \Re P(1,\xi) \in
L^{d}([s,t],\mathbb{C})$, equation \eqref{distorP} shows that
\[
\lbrack s,t]\ni\xi\mapsto\frac{\Re P(\phi_{s,\xi}(w),\xi )}{\Re
\phi_{s,\xi}(w)}
\]
belongs to $L^{d}([s,t],\mathbb{C}).$

Hence, the function $\phi(u):=\Re w\exp\left(  \int_{s}^{u}
\frac{\Re P(\phi_{s,\xi}(w),\xi)}{\Re \phi_{s,\xi
}(w)}d\xi\right)  $ is absolutely continuous in $[s,t]$ and
$\phi^{\prime}(u)=\phi(u)\frac{\Re P(\phi_{s,u}(w),u)} {\Re
\phi_{s,u}(w)}$ for almost every $u\in [s,t]$. By
\eqref{distorP}
\[
\frac{\partial}{\partial u}\left(  \frac{\phi(u)}{\Re \phi
_{s,u}(w)}\right)  \equiv 0
\]
for almost every $u\in\lbrack s,t]$. That is there exists $c$ such that $\Re %
\phi_{s,u}(w)=c\phi(u)$ for all $u.$ But, $\Re \phi
_{s,s}(w)=\Re w=\phi(s).$ That is, $c=1$ and the claim is
proved.
\end{proof}

\begin{corollary}
\label{DWpoint} Let $(\varphi_{s,t})$ be an evolution family of
the unit disc with Berkson-Porta data $(p,\tau)$. Suppose
$\tau$ is constant.  Then for all $0\leq s\leq t<+\infty$
either $\tau$ is the Denjoy-Wolff point of $(\v_{s,t})$ or
$\varphi_{s,t}=\id$.
\end{corollary}

\begin{proof}
It follows at once from Claim 1 and Claim 2 in the proof of
Theorem \ref{Continuidad-absoluta-multiplicador}.
\end{proof}

Our next result shows that a nice behavior of the derivative at
the common fixed point $\tau$ allows us to replace the
topological property EF3 in the definition of evolution family
by a much weaker hypothesis. In order to understand the
naturality of our hypothesis on the first derivative at the
Denjoy-Wolff point, we remark that if $(\v_{s,t})$ is an
evolution family on the unit disc with common Denjoy-Wolff
point $\tau\in \D$, then by univalence, $\v'_{0,t}(\tau)\neq 0$
for all $t\geq 0$. If $\tau\in\de\D$ then the same (as angular
limit) is true by the classical Julia lemma (see, {\sl e.g.},
\cite{Abate}).

\begin{theorem}
\label{EFwithcommonDW} Let $(\varphi_{s,t})$ be a family of
holomorphic self-maps of the unit disc  having a common
Denjoy-Wolff point $\tau \in\overline{\mathbb{D}}$. Assume that
$(\varphi_{s,t})$ satisfies EF1 and EF2 and  $\varphi
_{0,t}^{\prime}(\tau)\neq0$ for all $t\geq 0$. Then the
following are equivalent:

\begin{enumerate}
\item $(\varphi_{s,t})$ is an evolution family of order $d\geq 1$.

\item The following properties are satisfied:

\begin{enumerate}
\item[2.1] the map $\lbrack0,+\infty)\ni t\mapsto\mu(t):=\varphi_{0,t}^{\prime
}(\tau)$ is absolutely continuous and $\mu'\in
L_{loc}^{d}([0,+\infty),\mathbb{R})$,

\item[2.2] If $\tau\in\de \D$, there exists a point $z_{0}\in\mathbb{D}$ such that for all $T>0$
there exists a non-negative function $k_{z_{0},T}\in
L^{d}([0,T],\mathbb{R})$ such that
\[
|\varphi_{s,u}(z_{0})-\varphi_{s,t}(z_{0})|\leq\int_{u}^{t}k_{z_{0},T}%
(\xi)d\xi
\]
for all $0\leq s\leq u\leq t\leq T.$
\end{enumerate}
\end{enumerate}
\end{theorem}

\begin{proof}
By Theorem \ref{Continuidad-absoluta-multiplicador} and the
very definition of evolution family,  (1) implies (2).

Conversely, suppose (2) is satisfied. Again we have to split up
the inner and boundary cases.

Firstly, suppose that $\tau\in\mathbb{D}$. Up to conjugation,
we may assume  $\tau=0.$ Fix $0<T<+\infty.$ Since $\mu(t)\neq0$
for all $t$, there exist two absolutely continuous functions
$a,b:[0,+\infty )\rightarrow\mathbb{R}$ such that
$\mu(t)=e^{a(t)+ib(t)}$ for all $t$ and $a',b'\in
L_{loc}^{d}([0,+\infty),\mathbb{R})$.

By the chain rule for derivatives, $\varphi
_{s,t}^{\prime}(\tau)=\mu(t)/\mu(s)$ for all $0\leq s\leq
t<+\infty$ and since
$|\varphi_{s,t}^{\prime}(0)|=|\mu(t)/\mu(s)|=e^{a(t)-a(s)}\leq1,$
the map $a$ is decreasing.

Let $h_{s,t}(z):=\dfrac{\varphi_{s,t} (z)}{e^{i(b(t)-b(s))}z}$
for $z\in\mathbb{D}\setminus\{0\}$ and $h_{s,t}
(0):=e^{a(t)-a(s)}$ for all $0\leq s\leq t$. The map $h_{s,t}$
is holomorphic and $\Re (1-h_{s,t})\geq0.$ Therefore, by
\cite[pages 39-40]{Pommerenke},
\begin{align*}
|1-h_{s,t}(z)|  & \leq\frac{1+|z|}{1-|z|}|1-h_{s,t}(0)|=\frac{1+|z|}%
{1-|z|}(1-e^{a(t)-a(s)})\\
& =\frac{1+|z|}{1-|z|}e^{-a(s)}(e^{a(s)}-e^{a(t)})\leq\frac{1+|z|}%
{1-|z|}e^{-a(T)}(e^{a(s)}-e^{a(t)})
\end{align*}
whenever $z\in\mathbb{D}$ and $0\leq s\leq t\leq T.$

Now, fix $z\in\mathbb{D}$ and $0\leq s\leq u\leq t\leq T.$ On the one hand, if
$\varphi_{s,u}(z)=0,$ then
\[
|\varphi_{s,u}(z)-\varphi_{s,t}(z)|=|\varphi_{u,t}(\varphi_{s,u}%
(z))|=|\varphi_{u,t}(0)|=0.
\]
On the other hand, if $w:=\varphi_{s,u}(z)\neq0,$ then
\begin{align*}
|\varphi_{s,u}(z)-\varphi_{s,t}(z)|  & =|w-\varphi_{u,t}(w)|\leq\left\vert
1-\frac{\varphi_{u,t}(w)}{w}\right\vert \\
& \leq\left\vert 1-h_{u,t}(w)\right\vert +\left\vert \frac{\varphi_{u,t}%
(w)}{w}\right\vert \left\vert \dfrac{1}{e^{i(b(t)-b(u))}}-1\right\vert \\
& \leq\frac{1+|w|}{1-|w|}e^{-a(T)}(e^{a(u)}-e^{a(t)})+\left\vert
e^{ib(t)}-e^{ib(u)}\right\vert \\
& \leq\frac{1+|z|}{1-|z|}e^{-a(T)}(e^{a(u)}-e^{a(t)})+\left\vert
e^{ib(t)}-e^{ib(u)}\right\vert .
\end{align*}
In any case, we have that
\[
|\varphi_{s,u}(z)-\varphi_{s,t}(z)|\leq\frac{1+|z|}{1-|z|}e^{-a(T)}\left(
(e^{a(u)}-e^{a(t)})+\left\vert e^{ib(t)}-e^{ib(u)}\right\vert \right)  .
\]
 Then, the function
\[
k_{z,T}(\xi):=\frac{1+|z|}{1-|z|}e^{-a(T)}\frac{d}{d\xi}\left(
e^{a(\xi )}+\left\vert e^{ib(\xi)}\right\vert \right)
\]
belongs to $L^{d}([0,T],\mathbb{C}))$ and
\[
|\varphi_{s,u}(z)-\varphi_{s,t}(z)|\leq\int_{u}^{t}k_{z,T}(\xi)d\xi,
\]
proving the result in this case.

Next, suppose that $\tau\in\partial\mathbb{D}$. Again by the
chain rule for angular derivatives, $\varphi
_{s,t}^{\prime}(\tau)=\mu(t)/\mu(s)$ for all $0\leq s\leq
t<+\infty$ and, since $\varphi _{s,t}^{\prime}(\tau)\in(0,1]$,
it follows that $0<\mu (t)\leq\mu(s)\leq1$ for all $0\leq s\leq
t.$ Without loss of generality we assume $z_{0}=0.$ Again, we
move to the right half-plane by means of the Cayley transform
$T_{\tau}$ given by $T_{\tau} (z)=\frac{\tau+z}{\tau-z}$ and we
 let $(\phi_{s,t})$ be the family of holomorphic  self-maps
of the right half-plane given by
$\phi_{s,t}:=T_{\tau}\circ\varphi_{s,t}\circ T_{\tau}^{-1}$ for
all $0\leq s\leq t<+\infty.$ The Denjoy-Wolff point of
$\phi_{s,t}$ is $\infty$ with multiplier
$1/\varphi_{s,t}^{\prime} (\tau)=\mu(s)/\mu(t)$. Thus the
function $\Re[\phi_{s,t}(w)-\frac{\mu (s)}{\mu(t)}w]\geq0$ for
all $w\in \Ha$. Then, by \cite[pages 39-40]{Pommerenke},
\[
\left\vert \phi_{s,t}(w)-\frac{\mu(s)}{\mu(t)}w\right\vert \leq\frac
{|w+1|+|w-1|}{|w+1|-|w-1|}\left\vert \phi_{s,t}(1)-\frac{\mu(s)}{\mu
(t)}\right\vert .
\]
Therefore,
\begin{equation*}
|w-\phi_{s,t}(w)|   \leq\left\vert
1-\frac{\mu(s)}{\mu(t)}\right\vert |w|+\frac{|w+1|+|w-1|}
{|w+1|-|w-1|}\left\vert
\phi_{s,t}(1)-\frac{\mu(s)}{\mu(t)}\right\vert .
\end{equation*}
Fix $0<T<+\infty.$ By hypothesis and arguing as in the proofs
of Proposition \ref{EF-continuidad} and Lemma
\ref{EF-acotacion}, there is a number $R$ such that
$|\varphi_{s,t}(0)|\leq R$ for all $0\leq s\leq t\leq T$ and
then the set $\{\phi_{s,t}(1):$ $0\leq s\leq t\leq T\}$ is a
compact subset of the right half-plane.

Now fix $w\in\mathbb{H}$. If $0\leq s\leq t\leq T,$ then
$\rho_{\mathbb{H}}\left(  \phi_{s,t}(w),\phi_{s,t}(1)\right)
\leq \rho_{\mathbb{H}}\left(  w,1\right)$. Hence  there is a
compact set $K$ in $\mathbb{H}$ such that $\phi_{s,t}(w)\in K $
for all $0\leq s\leq t\leq T.$ Therefore there exists $M>0$
such that
\[
\max_{0\leq s\leq t\leq T}\left\{  |\phi_{s,t}(w)|,\frac{|\phi_{s,t}%
(w)+1|+|\phi_{s,t}(w)-1|}{|\phi_{s,t}(w)+1|-|\phi_{s,t}(w)-1|}\right\}  \leq
M.
\]

Fix $0\leq s\leq u\leq t\leq T.$ Write $v=\phi_{s,u}(w).$ Then
\begin{align*}
|\phi_{s,u}(w)-\phi_{s,t}(w)|  & =|v-\phi_{u,t}(v)|\leq\left\vert 1-\frac
{\mu(u)}{\mu(t)}\right\vert |v|+\frac{|v+1|+|v-1|}{|v+1|-|v-1|}\left\vert
\phi_{u,t}(1)-\frac{\mu(u)}{\mu(t)}\right\vert \\
& \leq M\left(  2\left\vert 1-\frac{\mu(u)}{\mu(t)}\right\vert +\left\vert
\phi_{u,t}(1)-1\right\vert \right)  \\
& \leq M\left(  \frac{2}{\mu(T)}\left\vert \mu(t)-\mu(u)\right\vert
+\frac{\left\vert \varphi_{u,t}(0)\right\vert }{1-\left\vert \varphi
_{u,t}(0)\right\vert }\right)  \\
& \leq M\left(  \frac{2}{\mu(T)}\left\vert \mu(t)-\mu(u)\right\vert +\frac
{1}{1-R}\left\vert \varphi_{u,t}(0)\right\vert \right)  .
\end{align*}
From this inequality, one can easily finish the proof arguing
as in the previous case.
\end{proof}

\begin{remark}
We point out that hypothesis 2.2 is not needed in case $\tau\in
\D$. While, in case $\tau\in\de \D$ hypothesis 2.2 in Theorem
\ref{EFwithcommonDW} cannot be removed. Indeed, the family
$\varphi_{s,t}(z)=T^{-1}\left( T(z)+ic(s)-ic(t)\right) $ where
$c$ is a continuous function from $[0,+\infty)$ into
$\mathbb{R}$ which is not absolutely continuous and
$T(z)=\frac{1+z}{1-z}$ satisfies EF1, EF2 and 2.1 (being
$\mu\equiv 1$) but it is not an evolution family.
\end{remark}

Classically, evolution families that comes out from Loewner
types equations are those $(\varphi_{s,t})$ with a common fixed
point $0$ and such that $\varphi_{0,t}^{\prime }(0)=e^{-t}$
(see, {\sl e.g.}, \cite{Loewner}, \cite{Pommerenke-65},
\cite{Pommerenke}, and \cite{Rosenblum-Rovnyak}). The above
result shows why it is not necessary to assume EF3 in this
classical case: it follows automatically from the normalization
hypothesis on the first derivative at~$0$.

We end up this section with a technical result which better
relates the classical definition of evolution family with the
definition introduced in this paper.

\begin{proposition}
Let $(\varphi_{s,t})$ be  an evolution family on the unit disc.
Then for all $r<1$ and for all $T<+\infty,$ the set of
functions $\{[0,T]\ni
t\mapsto\varphi_{0,t}(z)\in\mathbb{D}:|z|\leq r\}$ is uniformly
absolutely continuous.

Conversely, assume $(\varphi_{s,t})$ is a family of holomorphic
self-maps of the unit disc which satisfies EF1 and EF2 and has
a common Denjoy-Wolff point $\tau\in\mathbb{D}$. Assume
moreover that $\varphi_{0,t}^{\prime}(\tau)\neq0$ for all
$t\geq 0$. If for all $r<1$ and for all $T<+\infty,$ the set of
functions $\{[0,T]\ni
t\mapsto\varphi_{0,t}(z)\in\mathbb{D}:|z|\leq r\}$ is uniformly
absolutely continuous then $(\varphi_{s,t})$ is an evolution
family.
\end{proposition}

\begin{proof}
 By Lemma \ref{EF-acotacion}, there exists
$R>0$ such that $|\varphi_{s,t}(z)|\leq R$ for all $0\leq s\leq
t\leq T$ and $|z|\leq r$.  Let $G(z,t)$ be the Herglotz vector
field which solves \eqref{main-eq}, and let $k_{R,T}\in
L^{1}([0,T],\mathbb{R})$ be the function given by WHVF3. Then
\begin{align*}
\sup_{|z|\leq r}%
{\displaystyle\sum\limits_{k=1}^{n}} \left\vert
\varphi_{0,b_{k}}(z)-\varphi_{0,a_{k}}(z)\right\vert &
=\sup_{|z|\leq r}%
{\displaystyle\sum\limits_{k=1}^{n}} \left\vert
\int_{a_{k}}^{b_{k}}\frac{\partial\varphi_{0,\xi}(z)}{\partial\xi
}d\xi\right\vert \\
& \leq\sup_{|z|\leq r}%
{\displaystyle\sum\limits_{k=1}^{n}}
\int_{a_{k}}^{b_{k}}\left\vert G(\v_{0,\xi}(z),\xi)\right\vert d\xi\\
& \leq%
{\displaystyle\sum\limits_{k=1}^{n}}
\int_{a_{k}}^{b_{k}}k_{R,T}(\xi)d\xi.
\end{align*}

Conversely, assuming $\tau=0$, it is not  difficult to see that
the absolutely continuity and Cauchy formula imply that
 $\lbrack0,+\infty)\ni t\mapsto \varphi_{0,t}^{\prime}
(\tau)$ is absolutely continuous. Hence the result follows from
Theorem \ref{EFwithcommonDW}.
\end{proof}

\section{Evolution families with a common boundary fixed point}

In this section, we concentrate in the study of evolution
families $(\varphi_{s,t})$ with Denjoy-Wolff points a constant
$\tau\in
\partial\mathbb{D}$.

As we have remarked in the introduction, this case is much more
complicated and apart from a couple of papers due to Goryainov
and Ba \cite{Gorjainov} ,\cite{Goryainov-Ba}, there are no
references till the end of the nineties where a series of paper
of G.F. Lawler, O. Schramm,  W. Werner and Bauer appeared
\cite{Schramm}, \cite{Lawler-Schramm-Werner-I},
\cite{Lawler-Schramm-Werner-II}, \cite{Bauer}.

As usual, when the Denjoy-Wolff point is at the boundary, it is
better to translate to the right half-plane. Let $\Ha:=\{w\in
\C: \Re w>0\}$ be the right half-plane. As a matter of
notation, we say that a family $(\phi_{s,t})$ of holomorphic
self-maps of $\Ha$ is an {\sl evolution family of order $d\geq
1$} if there exists a biholomorphic map $T:\Ha\to \D$ such that
$(T\circ \phi_{s,t}\circ T^{-1})$ is an evolution family of
order $d$ in $\D$. Similar definition are given for Herglotz
vector fields and Herglotz functions.

Translating Theorems \ref{EF-implica-VF} and
\ref{Herglotz-implica-EF} to the right half-plane, we can state
the following result

\begin{theorem}\label{semipiano}
Let $(\phi_{s,t})$ be a family of holomorphic self-maps of the
right half-plane $\Ha$.  Then $(\phi_{s,t})$ is an evolution
family of order $d\geq 1$ in the right half-plane with $\infty$
as common boundary fixed point if and only if there exists a
Herglotz function $P(w,t)$ of order $d$ in the right half-plane
such that, given $s\geq0$, there exists a set $M=M(s)\subset
\lbrack s,+\infty)$ (not depending on $w$) of zero measure such
that, for every $t\in(s,+\infty)\setminus M$ and every
$w\in\mathbb{H}$, it holds that
\[
\frac{\partial\phi_{s,t}(w)}{\partial t}=P(\phi_{s,t}(w),t).
\]

Let $P$ be a Herglotz function of order $d\geq 1$ in the right
half-plane. For all $s\geq0$ and $w\in\mathbb{H},$ let
$\psi_{s,w}$ be the solution of the problem
\[
\left\{
\begin{array}
[c]{l}%
\overset{\bullet}{w}(t)=P(w(t),t)\quad \text{ for a. e. }t\in\lbrack s,+\infty)\\
w(s)=w.
\end{array}
\right.  \text{ }%
\]
Then defining $\phi_{s,t}(w):=\psi_{s,w}(t)$ for all $0\leq
s\leq t<+\infty$ and for all $w\in\mathbb{H},$ the family
$(\phi_{s,t})$ is an evolution family of order $d$ in the right
half-plane with $\infty$ as common boundary fixed point.
\end{theorem}

From the very beginning of this century, there have been many
authors interested on a very particular case of Herglotz
functions in the right half-plane (see, \cite{Schramm},
\cite{Marshall-Rohde}, \cite{Prokhorov-Vasiliev}). Namely, let
$h:[0,+\infty)\rightarrow i\mathbb{R}$ be a measurable function
(in fact, those papers always assume that $h$ is continuous).
Then $P(w,t)=\frac{1}{w+h(t)}$ is clearly a Herglotz function
in $\mathbb{H}$ of order $\infty$ since $|P(w,t)|\leq\frac
{1}{\Re w}$ for all $w\in\mathbb{H}$ and $t\geq0.$ Moreover, it
is clear that
$\angle\lim\limits_{w\rightarrow\infty}\frac{P(w,\xi)}{w}=0$
for all $\xi.$ Therefore, by Theorem \ref{semipiano}, if
$(\phi_{s,t})$ is the evolution family in the right half-plane
with Herglotz function $P,$ then all the functions $\phi_{s,t}$
are parabolic, that is,
$\angle\lim\limits_{w\rightarrow\infty}\frac{\phi_{s,t}(w)}{w}=1$.
By Claim 2 in the proof of Theorem
\ref{Continuidad-absoluta-multiplicador}, we have that for all
$w\in\mathbb{H}$\ and for all $0\leq s<t,$\ the function
\[
\lbrack s,t]\ni\xi\mapsto\frac{P(\phi_{s,\xi}(w),\xi)}{\phi_{s,\xi}(w)}%
\]
belongs to $L^{\infty}([s,t],\mathbb{C})$ and\textit{\ }%
\[
\phi_{s,t}(w)=w\exp\left(  \int_{s}^{t}\frac{P(\phi_{s,\xi}(w),\xi)}%
{\phi_{s,\xi}(w)}d\xi\right)  \text{ \quad for all }w\in\mathbb{H}.
\]
Now, write $k(w)=%
{\displaystyle\int_{s}^{t}}
\frac{P(\phi_{s,\xi}(w),\xi)}{\phi_{s,\xi}(w)}d\xi$ for all $w\in\mathbb{H}.$
Bearing in mind that $\phi_{s,\xi}$ is a parabolic function with $\infty$ as
Denjoy-Wolff point, if $w\in\mathbb{H}\cap\mathbb{R}$, then
\[
|\phi_{s,\xi}(w)|\geq\Re \phi_{s,\xi}(w)\geq w.
\]
Therefore,
\[
\left\vert w^{2}\frac{P(\phi_{s,\xi}(w),\xi)}{\phi_{s,\xi}(w)}\right\vert
\leq1
\]
for all $\xi$ and $\lim\limits_{w\in\mathbb{R},w\rightarrow+\infty}w^{2}%
\frac{P(\phi_{s,\xi}(w),\xi)}{\phi_{s,\xi}(w)}=1.$ Then, by the
dominated converge theorem, we have
\[
t-s=\lim\limits_{w\in\mathbb{R},w\rightarrow+\infty}w^{2}k(w)=\lim
\limits_{w\in\mathbb{R},w\rightarrow+\infty}w^{2}\left(  e^{k(w)}-1\right)
=\lim\limits_{w\in\mathbb{R},w\rightarrow+\infty}w(\phi_{s,t}(w)-w).
\]
Thus, by Lehto-Virtanen theorem, we obtain%
\[
\angle\lim_{w\rightarrow\infty}(\phi_{s,t}(w)-w)=0\quad\text{and}\quad
\angle\lim_{w\rightarrow\infty}w(\phi_{s,t}(w)-w)=t-s.
\]
That is,
\[
\phi_{s,t}(w)=w+\frac{t-s}{w}+\gamma_{s,t}(w)
\]
where $\angle\lim_{w\rightarrow\infty}w\gamma_{s,t}(w)=0$ and
the functions of the evolution family satisfies the so called
hydrodynamic normalization. Following the terminology
introduced by the last two authors and Pommerenke in
\cite{Contreras-Diaz-Pommerenke:Derivada-segunda}, this means
that if $(\varphi_{s,t})$ is the corresponding evolution family
in the unit disc  with fixed point $\tau,$ then there exist the
second and third angular derivatives
of $\varphi_{s,t}$ at $\tau$ and, in fact, $\varphi_{s,t}^{\prime\prime}%
(\tau)=0$ and $\varphi_{s,t}^{\prime\prime\prime}(\tau)=\frac{3}%
{2}(s-t)\overline{\tau}^{2}.$

\end{document}